\documentclass[12pt,oneside,final]{article}
\usepackage{amsthm}
\usepackage{amsmath}
\usepackage{amssymb}
\usepackage{latexsym}
\usepackage{bbm}
\usepackage[dvips]{graphicx}
\usepackage{enumerate}
\usepackage{verbatim}
\usepackage{amsfonts}
\usepackage{epsfig}
\usepackage{tikz}
\usetikzlibrary{fit}
\usetikzlibrary{shapes}
\usepackage{mathdots}
\usepackage{url}
\usepackage{subfigure}
\usepackage{overpic}
\usepackage{epstopdf}

\graphicspath{{Figures/}}

\newtheorem{thm}{Theorem}[section]
\newtheorem{prop}[thm]{Proposition}
\newtheorem{lem}[thm]{Lemma}
\newtheorem{cor}[thm]{Corollary}
\newtheorem{conj}[thm]{Conjecture}

\theoremstyle{definition}
\newtheorem{defn}[thm]{Definition}

\begin{document}




























\title{Isometries of the Hilbert Metric}
\author{Timothy Speer}
\date{}
\maketitle
\begin{abstract}
On any convex domain in $\mathbb{R}^n$ we can define the Hilbert metric. A projective transformation is an example of an isometry of the Hilbert metric. In this thesis we will prove that the group of projective transformations on a convex domain has at most index 2 in the group of isometries of the convex domain with its Hilbert metric. Furthermore we will give criteria for which the set of projective transformations between two convex domains is equal to the set of isometries of the Hilbert metric of these convex domains. Lastly we will show that $2$-dimensional convex domains with their corresponding Hilbert metrics are isometric if and only if they are projectively equivalent.

\end{abstract}

\tableofcontents


\section{Introduction}

On any open bounded convex set $\Omega$ in $\mathbb{R}^n$ we can define a metric $d_\Omega$ called the Hilbert metric. Any such convex set $\Omega$ will be called a convex domain and $(\Omega,d_\Omega)$ will be called a Hilbert geometry. The isometries of this metric have been studied in ~\cite{Cones}, ~\cite{Harpe}, and ~\cite{Polyhedral} when the convex domain is either a symmetric cone, strictly convex, or a polyhedral domain. In the remainder of this thesis we will investigate the nature of the isometries of the Hilbert metric for general convex domains and extend results known for the special types of domains listed above to general convex domains. After this work was complete the author found the work of Walsh in ~\cite{Walsh2} which proves one of the main results of this thesis using different techniques.

Suppose $\Omega$ is an open bounded convex set in $\mathbb{R}^n$. Given two distinct points $x$ and $y$ in $\Omega$ there exists a unique straight line containing both $x$ and $y$ and this line intersects $\partial\Omega$ in two points $\alpha, \beta$ satisfying $|\alpha-x|<|\alpha-y|$ and $|\beta-y|<|\beta-x|$.

\begin{figure}[h]
    \centering
    \def\svgwidth{\columnwidth}
    \begin{overpic}[]{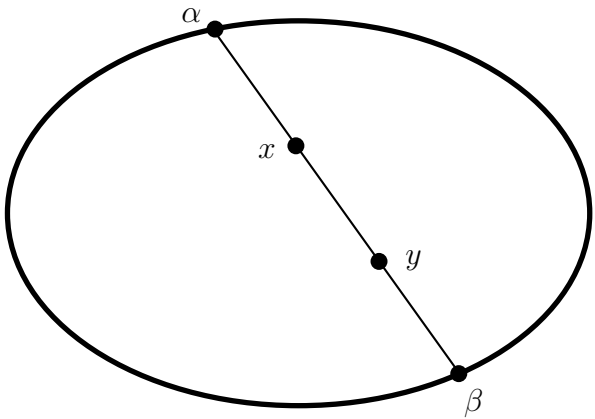}
        \put (30,66) {$\alpha$}
        \put (43,43) {$x$}
        \put (68,25) {$y$}
        \put (78,0) {$\beta$}
    \end{overpic}
    \vspace{3mm}
    \caption{The Hilbert metric on a convex domain $\Omega$.}
\end{figure}

\noindent The cross ratio of the fours points $\alpha,x,y,\beta$ is defined to be:
\[ CR(\alpha,x,y,\beta)=\frac{|\alpha-y|}{|\alpha-x|}\frac{|\beta-x|}{|\beta-y|}\]
\noindent which is always a number greater than one. The Hilbert metric on $\Omega$ is defined by
\begin{displaymath}
   d_\Omega(x,y) = \left\{
     \begin{array}{llr}
       \ln(CR(\alpha,x,y,\beta)) & if & x\neq y\\
       0 & if & x = y
     \end{array}
   \right.
\end{displaymath}

\noindent which makes $(\Omega, d_\Omega)$ into a complete metric space. The fact that $d_\Omega$ is a metric on $\Omega$ was shown by Hilbert in ~\cite{Hilbert} and can also be found in ~\cite{Harpe}. Hilbert introduced this metric as a generalization of the Cayley-Klein metric on hyperbolic space. If $\Omega$ is the interior of the unit ball in $\mathbb{R}^n$ then $(\Omega, d_\Omega)$ is isometric to the projective model of n-dimensional hyperbolic space. The Hilbert metric was used by Birkhoff in ~\cite{Birkhoff} to give a proof of the Perron-Frobenius theorem.

Recall that $\mathbb{RP}^n$ is the set of all lines through the origin in $\mathbb{R}^{n+1}$. We can think of $\Omega$ as lying in an affine subspace $\mathbb{R}^n$ which is contained in $\mathbb{RP}^n$. Denote the group of projective transformations of $\mathbb{RP}^n$ by $PGL_{n+1}(\mathbb{R})$. If $\Omega$ and $\Omega'$ are open bounded convex sets in $\mathbb{R}^n$ with the Hilbert metric then $Isom(\Omega,\Omega')$ is the set of isometries from $\Omega$ to $\Omega'$ and $PGL(\Omega,\Omega')=\{p\in PGL_{n+1}(\mathbb{R}): p(\Omega)=\Omega'\}$. In the case that $\Omega=\Omega'$ then these sets are groups and we will use the notation $Isom(\Omega)$ and $PGL(\Omega)$. It is well known that projective transformations preserve cross ratios and therefore give examples of isometries of the Hilbert metric. But in ~\cite{Harpe} de la Harpe gave an example of an isometry of the triangle which is a quadratic transformation and is not a projective transformation. He then asked for which convex domains $\Omega$ do the groups $PGL(\Omega)$ and $Isom(\Omega)$ coincide. In ~\cite{Walsh2} Walsh has shown these groups to be equal if all cones over $\Omega$ are not non Lorentzian symmetric cones.

Previous work has mainly be concerned with studying $PGL(\Omega)$ and its relationship to $Isom(\Omega)$. In this thesis we will consider these two groups as well as study isometries with possibly different domain and codomain. Two Hilbert geometries are \textit{projectively equivalent} if there is a projective transformation mapping one to the other. For all known cases, if two Hilbert geometries are isometric then they are projectively equivalent. We will show that this is always true in dimension 2.

A \textit{Lie group} is a manifold which is also a group such that the multiplication and inversion operations are smooth functions. $PGL(\Omega)$ is a Lie group because it is a closed subgroup of $PGL_{n+1}(\mathbb{R})$. Thus if $Isom(\Omega)=PGL(\Omega)$ then $Isom(\Omega)$ is a Lie group. In ~\cite{Harpe} de la Harpe conjectured that $Isom(\Omega)$ is always a Lie group. Corollary 1.4 in ~\cite{Walsh2} shows this to be true.

We will now state the main results obtained in this thesis describing the nature of $Isom(\Omega, \Omega')$ and its relation to $PGL(\Omega, \Omega')$. Theorem 1.1 tells us that $PGL(\Omega)$ is always a normal subgroup of $Isom(\Omega)$ and the quotient group is either the identity group or $\mathbb{Z}_2$. Corollary 1.4 in ~\cite{Walsh2} establishes the same result.

\begin{thm}
If $\Omega$ is an open bounded convex set in $\mathbb{R}^n$ then $PGL(\Omega)$ is a normal subgroup of $Isom(\Omega)$ and has index at most 2, that is $PGL(\Omega)=Isom(\Omega)$ or
\begin{displaymath}Isom(\Omega)/PGL(\Omega)\cong\mathbb{Z}_2.\end{displaymath}
\end{thm}
\noindent \textit{Sketch of proof:} The main steps in proving theorem 1.1 are the following:
\begin{enumerate}
\item Any isometry with a focusing n-simplex is projective
\item Any isometry between domains that contain extreme lines is projective
\item Any isometry with a focusing point is projective
\item Cones are only isometric to cones
\item An isometry maps a minimal cone to a minimal cone
\end{enumerate}

The most difficult steps to prove are 3 and 4. Once these 5 steps are established the next step is to show that if $\tau_1, \tau_2$ are isometries of $\Omega$ which are not projective transformations then $\tau_1 p=\tau_2 $ for some projective transformation $p$. Because $\tau_1$ and $\tau_2$ are not projective transformations they do not have focusing points. This means that $\tau_1^{-1}\tau_2$ has a focusing point and is a projective transformations.

Theorem 1.2 gives a useful criterion for proving that an isometry of the Hilbert metric is a projective transformation. For the definition of focusing point the reader should see the beginning of section 10. Note that we don't need to assume that the domain and codomain of the isometry are the same.

\begin{thm}
If $\tau:\Omega\rightarrow\Omega'$ is an isometry then $\tau$ is a projective transformation if and only if $\tau$ has a focusing point.
\end{thm}

The next theorem gives a characterization of a large class of convex domains for which all isometries are projective transformations. This class has a certain restriction placed on the structure of its boundary. As far as we know there is no name in the literature for this type of convex set. A point $e$ in $\partial\Omega$ is an {\em extreme point} of $\Omega$ if it is not contained in the interior of any line segment in the boundary of $\Omega$. An {\em extreme line} in $\Omega$ is a straight line in $\Omega$ whose closure intersects $\partial\Omega$ in extreme points of $\Omega$.

\begin{thm}
If $\tau:\Omega\rightarrow\Omega'$ is an isometry and $\Omega$ contains an extreme line then $\tau$ is a projective transformation.
\end{thm}

If $\Omega$ is an open bounded convex set in $\mathbb{R}^2$ that does not contain an extreme line then $\Omega$ is the interior of a triangle. The isometries of triangles were studied by de la Harpe in ~\cite{Harpe}. Thus Theorem 1.3 completely determines isometries in dimension 2 when $\Omega$ is not a triangle. We get the following characterization of isometries in dimension 2.

\begin{thm}
Suppose $\Omega$ and $\Omega'$ are the interiors of compact convex sets in $\mathbb{R}^2\subset\mathbb{R}P^2$ which are isometric. Then the following statements hold:

\begin{enumerate}

\item If $\Omega$ is not the interior of a triangle then $\Omega'$ is not the interior of a triangle and $Isom(\Omega, \Omega')=PGL(\Omega,\Omega')$.

\item (de la Harpe) If $\Omega$ is the interior of a triangle then $\Omega'$ is the interior of a triangle. Furthermore if $\Omega=\Omega'$ then $Isom(\Omega)\cong\mathbb{R}^2\rtimes D_6$ and $PGL(\Omega)\cong\mathbb{R}^2\rtimes D_3$ where $D_n$ is the dihedral group of order $2n$.

\item $\Omega$ and $\Omega'$ are isometric if and only if they are projectively equivalent.

\end{enumerate}

\end{thm}

In dimension $3$ the $3$-simplex and a cone on a disk have isometries which are not projective transformations. The next theorem characterizes $Isom(\Omega,\Omega')$ for 3-dimensional convex domains which are not cones.

\begin{thm}
Suppose $\Omega$ and $\Omega'$ are the interiors of compact convex sets in $\mathbb{R}^3\subset\mathbb{R}P^3$ which are isometric. If $\Omega$ is not a cone then $\Omega'$ is not a cone and $Isom(\Omega, \Omega')=PGL(\Omega,\Omega')$.
\end{thm}

It follows directly from theorem 1.5 that $3$-dimensional convex domains which are not cones are isometric if and only if they are projectively equivalent. The author conjectures that a statement similar to theorem 1.5 holds in all dimensions. In the next two sections we will give the background material on convex sets and the Hilbert metric that is needed to understand the remainder of this thesis.

\section{Convex Sets}

In this section we will give some basic properties of convex sets that are needed to understand the isometries of the Hilbert metric. All of these ideas are well known and further details and proofs can be found in ~\cite{Convex} and ~\cite{Convex2}. Recall that a set is convex if the straight line segment between any two points in the set is contained in the set. For our purposes it is crucial to understand the structure of the boundary of a closed convex set so we can understand the asymptotic behavior of complete geodesics of the Hilbert metric. For the remainder of this section $C$ will denote a convex set in $\mathbb{R}^n$. We start with a few definitions that apply to any convex set.

A subset $A$ of $\mathbb{R}^n$ is an \textit{affine subspace} of $\mathbb{R}^n$ if for all $x_1,x_2$ in $A$ and $\lambda_1,\lambda_2$ in $\mathbb{R}$ satisfying $\lambda_1+\lambda_2=1$ then $\lambda_1x_1+\lambda_2x_2$ is also in $A$. The intersection of affine subspaces is an affine subspace therefore any subset $S$ of $\mathbb{R}^n$ is contained in a minimal affine subspace. The minimal affine subspace containing $S$ is called the \textit{affine hull} of $S$ and is denoted by \textit{aff}$(S)$. The \textit{relative interior} of $C$ is the interior of $C$ in the affine hull of $C$ and the \textit{dimension} of $C$ is the dimension of the affine hull of $C$. The \textit{relative boundary} of $C$ is the boundary of $C$ in \textit{aff}$(C)$. Note that if $C$ is a convex domain in $\mathbb{R}^n$ then the relative interior and relative boundary of $C$ are just the usual interior and boundary of $C$ in $\mathbb{R}^n$.

For the remainder of this section we will assume that $C$ is a closed convex set unless otherwise mentioned. A convex subset $F\subset C$ is a \textit{face} of $C$ if for any two points $x,y$ in $C$ such that the open line segment $(x,y)$ intersects $F$ in at least one point then $[x,y]$ is contained in $F$. The entire set $C$ and the empty set are the improper faces of $C$ and any other face is a proper face of $C$. Recall that an extreme point of $C$ is a point in $C$ which is not contained in the interior of any open line segment in $C$. Another way to say this is that an extreme point of $C$ is a face of $C$ of dimension zero. A face of $C$ with dimension one less than the dimension of $C$ is called a \textit{facet} of $C$.

If $S$ is a set of faces of $C$ then we can define $inf(S)$ to be the intersection of the elements of $S$ and $sup(S)$ to be the intersection of the faces of $C$ that are contained in every element of $S$. If $\mathcal{F}(C)$ is the set of faces of $C$ then this set with inclusion and the above two operations forms a complete lattice. Corollary 5.7 of ~\cite{Convex} shows that the relative interiors of the faces in $\mathcal{F}(C)\setminus\emptyset$ partition $C$. In the case we care about, when $C$ is a convex domain, it follows that the relative interiors of the proper faces of $C$ partition the boundary of $C$.

A hyperplane $H$ in $\mathbb{R}^n$ is a \textit{supporting hyperplane} of $C$ if $C$ is contained in one of the closed halfspaces determined by $H$ and $H$ and $C$ are not disjoint. If $C$ is a non-empty closed convex set then it is the intersection of its supporting halfspaces, see theorem 4.5 of ~\cite{Convex}. This is known as the external representation of the convex set $C$. If $C$ is a compact convex set then it also has an internal representation as convex combinations of its extreme points, see theorem 5.10 of ~\cite{Convex}.

\begin{thm}[Krein-Milman Theorem]
If $C$ is a compact convex subset of $\mathbb{R}^n$ then $C$ is the convex hull of its extreme points.
\end{thm}

\noindent This theorem can be stated in further generality but this version is all we will need for our purposes.

\vspace{5mm}

\begin{figure}[h]
    \centering
    \def\svgwidth{\columnwidth}
    \begin{overpic}[]{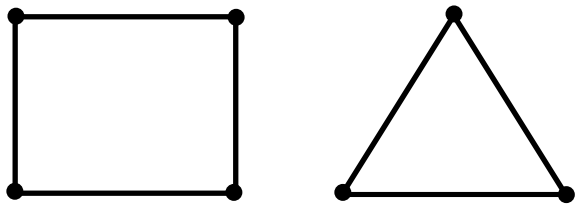}
    \end{overpic}
    \vspace{3mm}
    \caption{A convex set is the convex hull of its extreme points. In the case of polyhedral \noindent domains the extreme points are just the vertices.}
\end{figure}

From the Krein-Milman theorem we can deduce several properties of convex domains $\Omega$ that we will use throughout the rest of this thesis. First it guarantees that $\Omega$ has at least one extreme point. Furthermore, since $\Omega$ has non empty interior $\Omega$ contains at least $n+1$ extreme points in general position. The convex hull of these extreme points is a closed $n$-simplex whose interior is contained in $\Omega$. This idea will be used often enough that we will give it a name. An {\em extreme m-simplex} in $\Omega$ is an $m$-simplex whose vertices are extreme points of $\Omega$ and whose interior is contained in $\Omega$. With this terminology the Krein-Milman theorem guarantees that $\Omega$ contains an extreme $n$-simplex where $n$ is the dimension of $\Omega$.

If $v_1,...,v_n$ are $n$ points in general position then $[v_1,...,v_n]$ will be used to denote the closed $n$-simplex with vertices $v_1,...,v_n$ and $(v_1,...,v_n)$ will denote the interior of this $n$-simplex.

We will now describe some special types of convex domains whose Hilbert isometries were studied in ~\cite{Polyhedral} and ~\cite{Harpe}. A convex domain $\Omega$ is \textit{strictly convex} if all of its proper faces are extreme points. In ~\cite{Harpe} de la Harpe showed that for any strictly convex domain the only isometries of the Hilbert metric are projective transformations. A convex domain is a \textit{polyhedral domain} if it is the intersection of a finite number of open halfspaces. In ~\cite{Polyhedral} it was shown that for polyhedral domains all isometries are projective except in the case of the $n$-simplex. This was after de la Harpe who originally showed that a triangle has an isometry which is not a projective transformation. In sections 4 and 5 we will give further details and ideas used in proving these results.

\section{Properties of the Hilbert Metric}

In this section we give some properties of the Hilbert metric that are needed in the remainder of this thesis. If $\Omega$ is a convex domain in $\mathbb{R}^n$ then the topology induced on $\Omega$ from the Hilbert metric is equivalent to the standard topology on $\Omega$. Therefore when discussing ideas such as continuity and accumulation points we can use our understanding of the standard topology. The following definitions give some of the terminology we will use throughout this thesis.

\begin{defn}
$L$ is a \textit{line} in $\Omega$ if $L=l\cap\Omega$ for some straight line $l$ in $\mathbb{R}^n$ that has non-empty intersection with $\Omega$. The notation $[x,y]$ will denote the closed line segment between two points $x$ and $y$ and $(x,y)$ will denote the open line segment.
\end{defn}

\begin{defn}
$r$ is a \textit{ray} in $\Omega$ if $r=s\cap\Omega$ for some ray $s$ in $\mathbb{R}^n$ which begins at a point in $\Omega$. If $r$ is a ray in $\Omega$ we will denote its unique accumulation point in $\partial\Omega$ by $a(r)$.
\end{defn}

A \textit{geodesic} in $\Omega$ is a path in $\Omega$ which locally is distance minimizing. That is $\zeta:I\rightarrow\Omega$ is a \textit{geodesic} if for any $t\in I$ there exists an open set $U\subset I$ with $t\in U$ such that
\[ d_\Omega(\zeta(t_1),\zeta(t_2))=|t_1-t_2|\]
\noindent for all $t_1,t_2$ in $U$.

Since an isometry maps geodesics to geodesics we can study these maps by determining what they do to geodesics. In order to do this we need to know what the geodesics of the Hilbert metric are. If $x$ and $y$ are two points in a convex domain $\Omega$ then the line segment $[x,y]$ is a geodesic from $x$ to $y$. However, in general this is not the only geodesic from $x$ to $y$. Proposition 2 in ~\cite{Harpe} gives a criterion for determining if the straight line segment is a unique geodesic between two points. Since we will use this criterion throughout the thesis we state it here for convenience.

\begin{prop}[de la Harpe]
Let $\Omega$ be a convex domain in $\mathbb{R}^n$ and $x$ and $y$ be two points in $\Omega$. Furthermore, let $\alpha$ and $\beta$ be the points in the boundary of $\Omega$ where the straight line containing $x$ and $y$ intersects $\partial\Omega$. The points $\alpha$ and $\beta$ are contained in the relative interiors of faces $F_\alpha$ and $F_\beta$, respectively, in $\partial\Omega$. The following are equivalent:

\vspace{3mm}
i.) There exists a point $z$ in $\Omega$ which does not lie on the line containing $x$ and \indent $y$ and $d_\Omega(x,y)=d_\Omega(x,z)+d_\Omega(y,z)$.

\vspace{3mm}
ii.) There exists open line segments $l_\alpha$ and $l_\beta$ containing $\alpha$ and $\beta$ and contained \indent in the relative interior of $F_\alpha$ and $F_\beta$ respectively such that $l_\alpha$ and $l_\beta$ lie in a \indent two dimensional plane.

\end{prop}

It follows that $[x,y]$ is the unique geodesic from $x$ to $y$ if and only if for any open line segment $l_\alpha$ containing $\alpha$ and contained in $F_\alpha$ every open line segment containing $\beta$ and contained in $F_\beta$ is skew to $l_\alpha$.

\begin{figure}[h]
    \centering
    \def\svgwidth{\columnwidth}
    \begin{overpic}[]{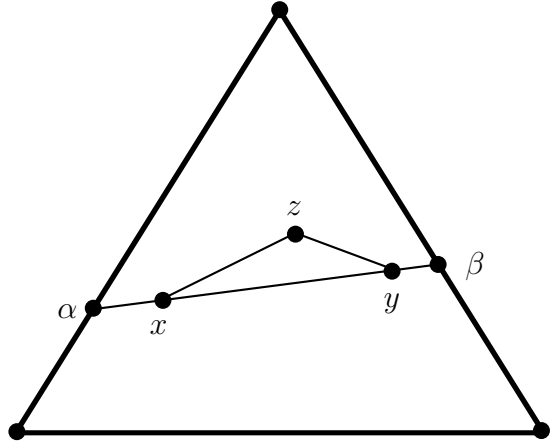}
        \put (26,20) {$x$}
        \put (69,25) {$y$}
        \put (51,42) {$z$}
        \put (9,23) {$\alpha$}
        \put (84,31) {$\beta$}
    \end{overpic}
    \vspace{3mm}
    \caption{The line segment $[x,y]$ is not a unique geodesic from $x$ to $y$ in this triangle because we can find open line segments around $\alpha$ and $\beta$ in the boundary of $\Omega$ which are coplanar. The path obtained by going from $[x,z]$ to $[z,y]$ is also a geodesic from $x$ to $y$. This is because $d_\Omega(x,y)=d_\Omega(x,z)+d_\Omega(y,z)$. Because geodesics are not unique it is possible to have an isometry which is not a projective transformation.}
\end{figure}

Using the criterion in proposition 3.3 we can construct convex domains that have points which do not have a unique geodesic between them. Hence the image of a straight line under a Hilbert isometry could be a geodesic which is not a line. Thus we will give a special name to those lines and rays for which geodesics between two points on the line are unique.

\begin{defn}
A straight line L in $\Omega$ is {\em rigid} if for any two points $a$ and $b$ on $L$, the line segment $[a,b]$ is the unique geodesic from $a$ to $b$. This same idea can be used to define terms such as rigid line segment and rigid ray as well.
\end{defn}

Rigid lines are useful because we know how they behave under isometries and understanding their behavior is the key to proving a given isometry is a projective transformation.

\begin{lem}
Suppose that $\tau:(\Omega,d_{\Omega})\rightarrow(\Omega',d_{\Omega'})$ is an isometry. If $L$ is a rigid line (ray) in $\Omega$ then $\tau(L)$ is a rigid line (ray) in $\Omega'$.
\end{lem}
\begin{proof}
Geodesics are preserved by isometries so it follows from the definition of rigid line that $\tau(L)$ is a rigid line in $\Omega'$.
\end{proof}

A line in $\Omega$ has two accumulation points in the boundary of $\Omega$. If at least one of these points is an extreme point of $\overline{\Omega}$ then this line is a rigid line because of the criterion in proposition 3.3. But this is not the only example of a rigid line. For example, if $\Omega$ is a 3-simplex any line in $\Omega$ between pairs of skew edges in the boundary of $\Omega$ is a rigid line.

A \textit{complete geodesic} in $\Omega$ is a map $\zeta:\mathbb{R}\rightarrow\Omega$ such that for any two points $t_1, t_2\in\mathbb{R}$ the path from $t_1$ to $t_2$ determined by $\zeta$ is a geodesic. A line in $\Omega$ is an example of a complete geodesic. In general complete geodesics in $\Omega$ have exactly two accumulation points in $\partial\Omega$ and geodesic rays must have an accumulation point in $\partial\Omega$. For proofs and more details see ~\cite{Minkowski}. There are convex domains with complete geodesics whose accumulation points in $\partial\Omega$ are extreme points and yet the geodesic is not a line. For example let $x$ be a point in a triangle. Then the path formed by taking two rays starting at $x$ to two of the vertices of the triangle is a complete geodesic whose accumulation points are extreme points of the triangle.

The next theorem is a well known result about when the cross ratios of two groups of four collinear points are the same.

\begin{thm}[Cross Ratio Theorem]
Suppose $p$ is a point in $\mathbb{R}^2$ and $L$ is a line through $p$. Let $r_i$ for $i=1,..,4$ be distinct rays starting at $p$ lying in the same halfspace determined by $L$ as in figure 4. If $\alpha_i$ are collinear points on $r_i$ and $\beta_i$ are collinear points on $r_i$ then
\[CR(\alpha_1,\alpha_2,\alpha_3,\alpha_4)=CR(\beta_1,\beta_2,\beta_3,\beta_4).\]
\end{thm}

\vspace{3mm}
\begin{figure}[h]
    \centering
    \def\svgwidth{\columnwidth}
    \begin{overpic}[]{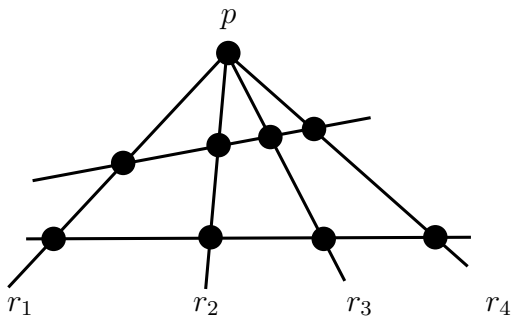}
        \put (46,58) {$p$}
        \put (0,-4) {$r_1$}
        \put (40,-4) {$r_2$}
        \put (73,-4) {$r_3$}
        \put (103,-4) {$r_4$}
    \end{overpic}
    \vspace{3mm}
    \caption{Cross Ratio Theorem}
\end{figure}

In the remainder of this section we will give an alternative definition of the Hilbert metric which is due to Birkhoff. To prove the results of this thesis we will not use this definition however this second definition was used in ~\cite{Polyhedral} to study isometries of polyhedral domains. The reader should consult ~\cite{Polyhedral}, ~\cite{Walsh2} and ~\cite{Birkhoff} for further details on this second definition of the Hilbert metric.

An open convex cone $C$ is \textit{proper} if $\overline{C}\cap-\overline{C}=\{0\}$. If $C$ is a proper open convex cone then we can define a partial order $\leq_C$ on $\mathbb{R}^{n+1}$ by $x\leq_C y$ if $x-y$ is in $\overline{C}$. Then we define the following for $x\in C$ and $y\in\mathbb{R}^{n+1}$:
\[ M(y/x,C)=inf\{\lambda>0: y\leq_C\lambda x\} \]

\noindent The Hilbert metric on $C$ is defined to be:

\vspace{3mm}
\begin{center}
$d_C(x,y)=ln(M(y/x,C))+ln(M(x/y,C))$ for $x,y\in C$
\end{center}

\vspace{3mm}
\noindent Since $C$ is a proper open cone this gives a well defined metric on the rays of $C$. If $\Omega$ is an $n$-dimensional convex domain in $\mathbb{R}^{n+1}$ which does not contain the origin then let $C_\Omega$ be the proper open convex cone obtained by taking all the rays starting from the origin in $\mathbb{R}^{n+1}$ that pass through a point in $\Omega$. We can see that $(\Omega,d_\Omega)$ and $(C_\Omega,d_{C_\Omega})$ are isometric because $d_\Omega$ and $d_{C_\Omega}$ agree on $\Omega$, see ~\cite{Birkhoff}. This alternative definition of the Hilbert metric is sometimes referred to as Hilbert's projective metric to differentiate it from the original definition.

\section{$Isom(\Omega)$ for Strictly Convex Domains}

The first type of convex domains whose isometries were studied was the strictly convex case. Recall that a convex domain $\Omega$ is strictly convex if all the proper faces of $\Omega$ are extreme points. An example of a strictly convex domain is an ellipsoid which we have already seen is isometric to the projective model for hyperbolic space. In ~\cite{Harpe} de la Harpe proved the following theorem about the isometry group of $\Omega$ when $\Omega$ is strictly convex.

\begin{thm}[de la Harpe]
If $\Omega$ is a strictly convex domain then
\[Isom(\Omega)=PGL(\Omega).\]
\end{thm}

The assumption that $\Omega$ is strictly convex guarantees that geodesics in $\Omega$ are unique due to proposition 3.3. Using our terminology, every line in $\Omega$ is rigid. This means that the image of any line in $\Omega$ under an isometry is a line. The main idea in the proof of theorem 4.1 was to show that because an isometry maps lines to lines it can be extended to $\overline{\Omega}$ and therefore preserves cross ratios of collinear points in $\overline{\Omega}$.

A \textit{projective basis} for $\mathbb{RP}^n$ is a set of $n+2$ points in $\mathbb{RP}^n$ such that each subset of $n+1$ of these points is linearly independent. To finish the proof of theorem 4.1 de la Harpe shows that if we choose a projective basis in $\Omega$ then its image under any isometry of $\Omega$ is a projective basis. There exists a projective transformation that agrees with the given isometry on this basis. By extending the region on which the isometry and projective transformation agree on de la Harpe shows that in fact they must be the same map. For further details see the proof of proposition 3 in ~\cite{Harpe}.

Using the results of this thesis we obtain a second proof of this theorem. Any strictly convex domain $\Omega$ must contain an extreme line because all the proper faces of $\Omega$ are extreme points. Thus it follows from theorem 1.3 that any isometry of $\Omega$ is a projective transformation. We have the following generalization of de la Harpe's theorem.

\begin{thm}
If $\tau:\Omega\rightarrow\Omega'$ is an isometry and $\Omega$ is strictly convex then $\Omega'$ is strictly convex and $\tau$ is a projective transformation.
\end{thm}
\begin{proof}
Because $\Omega$ contains an extreme line it follows from theorem 1.3 that $\tau$ is a projective transformation. Since every line in $\Omega$ is a an extreme line it follows from lemma 9.8 that every line in $\Omega'$ is an extreme line and therefore $\Omega'$ is strictly convex.
\end{proof}

From theorem 4.2 we have that  two strictly convex domains are isometric if and only if they are projectively equivalent and that a strictly convex domain can not be isometric to a convex domain which is not strictly convex.

\section{$Isom(\Omega)$ for Polyhedral Domains}

A convex domain $\Omega$ is polyhedral if it is the intersection of a finite number of open halfspaces. In ~\cite{Polyhedral} Lemmens and Walsh proved that a polyhedral domain has an isometry which is not a projective transformation if and only if it is an $n$-simplex. The special case of the triangle was studied first by de la Harpe in ~\cite{Harpe}. Some of the ideas used by de la Harpe to determine the isometry group for the triangle will be generalized and used throughout the thesis so we will now give an overview of his work.

Since any two triangles are projectively equivalent we will work with the standard open 2-simplex \[\Delta_2=\{(x,y,z)\in\mathbb{R}^3:x>0,y>0,z>0, x+y+z=1\}.\] We will first give de la Harpe's description of $PGL(\Delta_2)$. Given a diagonal matrix
\[ M=\left( \begin{array}{ccc}
\lambda_1 & 0         & 0 \\
0         & \lambda_2 & 0 \\
0         & 0         & \lambda_3
\end{array} \right) \]

\noindent in $PGL_3(\mathbb{R})$ with positive diagonal entries we get a map of $\Delta_2$ defined by
\[ M\left( \begin{array}{c}
x  \\
y   \\
z
\end{array} \right) = \frac{1}{\lambda_1x+\lambda_2y+\lambda_3z}\left( \begin{array}{c}
\lambda_1x  \\
\lambda_2y   \\
\lambda_3z
\end{array} \right)\]

\noindent If $T_2$ is the group of equivalence classes of all such matrices $M$ then this definition gives a transitive action of $T_2$ on $\Delta_2$. Furthermore the equivalence classes in $PGL_3(\mathbb{R})$ of the six permutation matrices act on $\Delta_2$. It follows that $PGL(\Delta_2)=T_2\rtimes S_3$ where $S_3$ is the symmetric group on 3 elements.

To show that $\Delta_2$ has an isometry which is not a projective transformation de la Harpe showed that there is an isometry between $\Delta_2$ and a normed vector space known as the hexagonal plane. Given a finite dimensional real vector space $V$ and a symmetric convex domain $K$ in $V$ containing the origin, the Minkowski functional for $K$ is the function $p_K:V\rightarrow [0,\infty)$ defined by
\[p_k(v)=inf\{\lambda\in\mathbb{R}: \lambda>0, v\in\lambda K\}.  \]

\noindent It is well known that $p_K$ is a norm on $V$. The hexagonal plane is the normed vector space $W_2=\{(x,y,z)\in\mathbb{R}^3: x+y+z=0\}$ equipped with the norm arising from the Minkowski functional $p_K$ where $K$ is the regular hexagon containing the origin in $W_2$ whose edges are determined by the equations $|x-y|=1$, $|x-z|=1$ and $|y-z|=1$. It follows that the unit ball of the hexagonal plane is a hexagon. de la Harpe defined a function $\Lambda_2:\Delta_2\rightarrow W_2$ by
\[\Lambda_2(x,y,z)=\frac{1}{3}(ln(\frac{x^2}{yz}),ln(\frac{y^2}{xz}),ln(\frac{z^2}{xy}))\]

\noindent and showed that $\Lambda_2$ is an isometry from $\Delta_2$ with the Hilbert metric to the hexagonal plane, see ~\cite{Harpe} proposition 7 for more details. It follows that $Isom(\Delta_2)=\mathbb{R}^2\rtimes D_6$ where $D_6$ is the group of symmetries of a hexagon. Notice that $PGL(\Delta_2)$ is an index 2 subgroup of $Isom(\Delta_2)$. Theorem 1.1 shows that this is the only possibility for general convex domains which have non projective isometries.

The function $f:W_2\rightarrow W_2$ defined by $f(v)=-v$ is an isometry of the hexagonal plane and so $\gamma=\Lambda_2^{-1}f\Lambda_2$ is an isometry of $\Delta_2$ with the Hilbert metric. The map $\gamma$ is the first example of an isometry of the Hilbert metric which is not a projective transformation. In addition, every non projective isometry of $\Delta_2$ is obtained by composing $\gamma$ with a projective transformation of $\Delta_2$.

As was mentioned in section 3 of this thesis there are points in $\Delta_2$ which contain many geodesics between them and because of this $\Delta_2$ can have isometries like $\gamma$ which are not projective. The image under $\gamma$ of any straight line in $\Delta_2$ which passes through the center of $\Delta_2$ is a straight line if and only if it is a rigid line, see proposition 8 in ~\cite{Harpe}. Since projective transformations map straight lines to straight lines it follows that $\gamma$ is not projective. The domain $\Delta_2$ contains three families of rigid lines, one for each vertex. $\gamma$ maps each of these families of lines to the same family of lines with the orientation reversed. This idea is important for understanding the general behavior of isometries of the Hilbert metric in any dimension as will be seen later in this thesis.

\vspace{5mm}
\begin{figure}[h]
    \centering
    \def\svgwidth{\columnwidth}
    \begin{overpic}[]{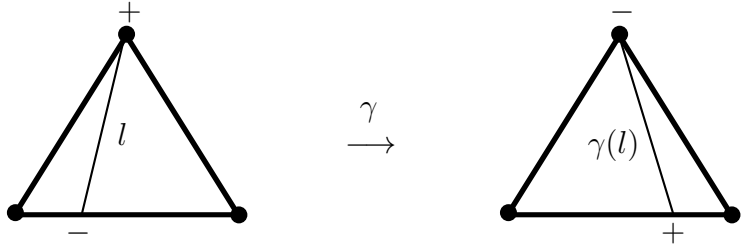}
        \put (46,10) {$\longrightarrow$}
        \put (48,15) {$\gamma$}
        \put (15,28) {$+$}
        \put (8,-2) {$-$}
        \put (89,-2) {$+$}
        \put (82,28) {$-$}
        \put (15,11){$l$}
        \put (79,10){$\gamma(l)$}
    \end{overpic}
    \vspace{3mm}
    \caption{$\gamma$ maps the three families of lines through the vertices of $\Delta_2$ to themselves but with reversed orientation.}
\end{figure}

Also note that $\gamma^2$ is the identity transformation. We will use this idea later that show that the square of any Hilbert isometry is a projective transformation. It turns out that $\gamma$ has the form of a quadratic transformation which in projective coordinates is given by
\[\gamma[x:y:z]=[x^{-1}:y^{-1}:z^{-1}].\]

\noindent From this formula it can be seen that $\gamma$ blows up the vertices of $\Delta_2$ to edges and blows down the edges of $\Delta_2$ to vertices, see ~\cite{Harpe} proposition 8 for more details.

One can generalize the idea of the hexagonal plane by defining \[W_n=\{(x_1,...,x_{n+1})\in\mathbb{R}^{n+1}: x_1+...+x_{n+1}=0\}\] to be a normed space analogous to the hexagonal plane. Then we get an isometry $\Lambda_n:\Delta_n\rightarrow W_n$ where $\Delta_n$ is the $n$-simplex with the Hilbert metric. de la Harpe did not determine the isometry group for the $n$-simplex in ~\cite{Harpe} but he gave the definition of $W_n$ and $\Lambda_n$.

In addition to the case of the triangle he showed that all isometries of a square are projective and its isometry group is isomorphic to $D_4$. Furthermore he showed that the interior of a 2-dimensional polyhedral domain has a finite isometry group if it has more that three vertices, see proposition 4 in ~\cite{Harpe}. Most of these results have been subsumed by the work of Lemmens and Walsh in ~\cite{Polyhedral} as we will see below.

Note that the existence of $\Lambda_n$ shows that the interior of an $n$-simplex with the Hilbert metric is isometric to a normed space. But the converse of this is also true, that is if $\Omega$ is a convex domain with the Hilbert metric which is isometric to a normed space then $\Omega$ must be the interior of an $n$-simplex. For the proof of the converse see ~\cite{Minkowski}.

The relationship between $Isom(\Omega)$ and $PGL(\Omega)$ for general polyhedral domains was completely determined by Lemmens and Walsh in ~\cite{Polyhedral}. They proved the following theorem
\begin{thm}[Lemmens, Walsh]
If $\Omega$ is a convex polyhedral domain with the Hilbert metric then
\[Isom(\Omega) = PGL(\Omega)\]

\noindent if and only if $\Omega$ is not an n-simplex. Furthermore if $\Omega$ is an n-simplex then
\[ PGL(\Omega)\cong\mathbb{R}^{n}\rtimes\sigma_{n+1} \hspace{7mm}and\hspace{7mm} Isom(\Omega)\cong\mathbb{R}^n\rtimes\Gamma_{n+1} \]

\noindent where $\sigma_{n+1}$ is the group of coordinate permutations of $\mathbb{R}^{n+1}$ and $\Gamma_{n+1}$ is the product of $\sigma_{n+1}$ with the group generated by the function $x\rightarrow -x$.

\end{thm}

We will now give a general overview of the main ideas and concepts used by Lemmens and Walsh to prove Theorem 5.1. For definitions and further details the reader should consult ~\cite{Polyhedral}, ~\cite{Walsh}, and ~\cite{Rieffel}.

In ~\cite{Walsh} Walsh defines the horofunctions and horoboundary of a convex domain with the Hilbert metric using a method originally given by Gromov in ~\cite{Gromov}. A Busemann point is a horofunction which is the limit of an almost geodesic, see ~\cite{Rieffel}. In particular every geodesic in $\Omega$ determines a Busemann point.   Walsh gives a characterization of the Busemann points for the Hilbert metric, see theorem 1.1 ~\cite{Walsh}. Lemmens and Walsh show that there is a metric on the set of Busemann points of a Hilbert geometry called the detour metric which is possibly infinite. Given an isometry between two convex domains with the Hilbert metric the distance between Busemann points given by the detour metric is also preserved.

A \textit{part} is a subset of the set of Busemann points on which the detour metric is finite. To each part they assign a face of $\Omega$. If this face is a vertex or facet of $\Omega$ then the part is called a vertex part or facet part. The key to finishing the proof is to show that any isometry of polyhedral domains either sends vertex parts to vertex parts and facet parts to facet parts or it interchanges the two. If vertex parts are sent to vertex parts then the isometry extends continuously to the boundary of the Hilbert domain and it follows that the isometry is a projective transformation. If the isometry interchanges vertex parts and facet parts then the polyhedral domain must be an $n$-simplex, see theorems 6.1 and 7.1 in ~\cite{Polyhedral}.

To prove the results in this thesis we will not be using the detour metric or Busemann points, however the idea of focusing points which are introduced in section 8 are in some sense an attempt to extend this idea of interchanging vertex and facet parts to a general convex domain using extreme points and rigid lines.

\section{$Isom(\Omega)$ for Symmetric Cones}

The isometries of the Hilbert metric for symmetric convex cones were studied by Molnar in ~\cite{Molnar} and Bosche in ~\cite{Cones}. Molnar studied the case of the symmetric cone of positive-definite Hermitian matrices and Bosche extended his results to general symmetric cones.

A convex domain $\Omega$ is an \textit{open convex cone} if for all $\lambda\in\mathbb{R}_{+}$ we have $\lambda\Omega\subset\Omega$. A convex cone is \textit{proper} if it does not contain a complete line. The \textit{dual cone} of $\Omega$ is the following set
\[ \Omega^*=\{x\in\mathbb{R}^n: \langle x,y\rangle>0, \forall y\in\Omega \}\]

\noindent which is also a cone. For our purposes a \textit{homogeneous cone} is an open proper convex cone $\Omega$ for which $PGL(\Omega)$ acts transitively on $\Omega$. A homogeneous cone is a \textit{symmetric cone} if in addition it is equal to its dual cone. For more details on homogeneous and symmetric cones the reader should see ~\cite{Cones} and ~\cite{Kai}. Bosche proved the following result about isometries of the Hilbert metric for symmetric cones.

\begin{thm}[Bosche]
If $\Omega$ is symmetric cone then $PGL(\Omega)$ is a normal subgroup of $Isom(\Omega)$ and has index at most 2. Furthermore
\[PGL(\Omega)=Isom(\Omega)\]

\noindent if and only if $\Omega$ is a Lorentz cone.
\end{thm}

The Lorentz cone in dimension $n$ is the set of points satisfying the inequality $x_1^2 > x_2^2+...+x_n^2$. The projectivization of this cone is just an ellipsoid of dimension $n-1$ and is a model for hyperbolic $(n-1)$-space. An ellipsoid is strictly convex so this part of the theorem adds nothing new to the results of de la Harpe in ~\cite{Harpe} on strictly convex domains.

However the theorem gives us the first examples of convex domains, besides $n$-simplices, which have isometries that are not projective transformations. For example, the cone on a disk is a non Lorentzian symmetric cone, see ~\cite{Harpe} proposition 5, so it has an isometry which is not a projective transformation. Any rigid line through the vertex of this cone is mapped by this isometry to a rigid line through the vertex but with the orientation reversed.

In ~\cite{Polyhedral} Lemmens and Walsh conjectured that $\Omega$ has an isometry which is not projective if and only if a cone over $\Omega$ is a non Lorentzian symmetric cone and these isometries are generated by Vinberg's $*-map$. This conjecture was established by Walsh in ~\cite{Walsh2}.

On any homogeneous cone we can define Vinberg's $*-map$ as follows, let $\phi:\Omega\rightarrow\mathbb{R}_+$ be the characteristic function on $\Omega$ defined by
\[\phi(x)=\int_{\Omega^*}e^{-\langle x,y\rangle}dy \]

\noindent where $y$ is the Euclidean measure. Then Vinberg's $*-map$ is the function from $\Omega$ to $\Omega^*$ which maps $x$ to $x^*$ where
\[ x^*=-\nabla log\phi(x). \]

To prove theorem 6.1 Bosche used the well known idea of associating a Jordan algebra to a symmetric cone. The Jordan algebra of a symmetric cone has an inverse map which is an involution. This inverse map coincides with Vinberg's $*-map$ under a suitable interpretation. Since $\Omega=\Omega^*$ this map gives us an isometry of $\Omega$ which is an involution. In the case of a Lorentz cone this map is a projective transformation. Since we will not be working with Jordan algebras for the remainder of this thesis the reader should see ~\cite{Kai} and ~\cite{Cones} for further information on Vinberg's $*-map$ and the Jordan algebra associated to a symmetric cone and then end of ~\cite{Polyhedral} for more on how this relates to isometries of the Hilbert metric.

\section{Cross Sections of $\Omega$}

In this section we will start studying the Hilbert metric on a general convex domain. We begin by giving the definition for a cross section of $\Omega$ and some of the important properties of these objects.

\begin{defn}
Let $P$ be an affine subspace in $\mathbb{R}^n$ with $dim(P)=m$ and $2\leq m\leq n$. If $D=P\cap\Omega$ is non empty then $D$ is an $m$-dimensional \textit{cross section of $\Omega$}. If $D\neq\Omega$ then $D$ is a proper cross section of $\Omega$.
\end{defn}

\begin{lem}
If $D$ is an m-dimensional cross section of $\Omega$ then $D$ is a convex domain of dimension $m$ in $P$ and the relative boundary of $D$ is contained in $\partial\Omega$.
\end{lem}
\begin{proof}
$D$ is convex because it is the intersection of two convex sets and $D$ is bounded because $\Omega$ is bounded. If $x$ is in $D$ then there exists an open set $U$ of $\mathbb{R}^n$ containing $x$ with $U\subset\Omega$. Hence $P\cap U\subset D$ is an open subset of $P$ which contains $x$. Because $D$ is an open subset of $P$ is must have dimension $m$, so we are done with the first part of the lemma.

If $\alpha$ is in the relative boundary of $D$ then the ray $[x,\alpha)$ is contained in $D$ and therefore is in $\Omega$. Hence $\alpha$ is an accumulation point of $\Omega$ which is not in $\Omega$ and therefore is in $\partial\Omega$.
\end{proof}

\begin{lem}
If $D$ is a cross section of $\Omega$ then $d_D=d_\Omega$ for points in $D$.
\end{lem}
\begin{proof}
Let $x$ and $y$ be two points in $D$ and $\alpha$ and $\beta$ be the two points where the line between $x$ and $y$ intersects the relative boundary of $D$. By lemma 7.2 the points $\alpha$ and $\beta$ are in $\partial\Omega$. It follows from the definition of the Hilbert metric that $d_D(x,y)=d_\Omega(x,y)$.
\end{proof}

\section{Asymptotic Geometry}

In this section we will study some of the asymptotic geometry of the Hilbert metric. In particular it is important to understand the behavior of rays and geodesics as one approaches the boundary of a convex domain $\Omega$. We have seen that if two rigid rays converge to the same point in $\partial\Omega$ then it is possible that their images under a Hilbert isometry converge to different points in $\partial\Omega'$. Understanding this behavior will be the main focus of sections 8, 9 and 10.

We will now give a criterion for determining when two distinct points in $\partial\Omega$ lie in the relative interior of the same face of $\Omega$.

\begin{lem}
Two distinct points $\alpha$ and $\beta$ in $\partial\Omega$ lie in the relative interior $I$ of the same face of $\Omega$ if and only if there exists an open line segment $S\subset \partial\Omega$ with $\alpha$ and $\beta$ in $S$.
\end{lem}
\begin{proof}
First suppose $\alpha$ and $\beta$ lie in $I$. From theorem 3.4 in ~\cite{Convex} $I$ is a convex set and thus $[\alpha,\beta]$ is contained in $I$. Since $I$ is an open set in its affine hull and $dim(I)\geq1$ we can find an open segment $S\supset[\alpha,\beta]$ contained in $I$.

Now assume that there is an open line segment $S$ contained in $\partial\Omega$ and containing $\alpha$ and $\beta$. Let $F$ and $F'$ be the proper faces of $\overline{\Omega}$ whose relative interiors contain $\alpha$ and $\beta$ respectively. We will show that $F=F'$. Theorem 5.6 in ~\cite{Convex} tells us that $F$ is the smallest face of $\Omega$ containing $\alpha$ and $F'$ is the smallest face of $\Omega$ containing $\beta$. Since $S$ intersects $F$ it follows from the definition of face that $S$ and thus $\beta$ is contained in $F$ and so $F'\subset F$. The opposite inclusion follows from the same reasoning.
\end{proof}

The relative interiors $I$ and $I'$ of two faces of $\Omega$ are \textit{opposite} if there is a line in $\Omega$ from $I$ to $I'$. Throughout this thesis will make use of the fact that the relative interior of an extreme point is itself.

\begin{lem}
If $I$ and $I'$ are the relative interiors of faces of $\Omega$ which are opposite then any line from $I$ to $I'$ is contained in $\Omega$.
\end{lem}
\begin{proof}
Let $L$ be a line in $\Omega$ that converges to $\alpha_1$ in $I$ and $\beta_1$ in $I'$. Let $\alpha_2$ be a point in $I$ and $L'$ the line from $\alpha_2$ to $\beta_1$. If $P$ is the two dimensional plane containing $L$ and $L'$ then $P\cap\Omega$ is a two dimensional convex domain which is contained in $\Omega$.

It follows from lemma 8.1 that there is an open segment $S\subset\partial\Omega$ which contains $\alpha_1$ and $\alpha_2$. Thus $\alpha_1$ and $\alpha_2$ lie in the relative interior of a 1-dimensional face of $P\cap\Omega$. Because $L$ is in $P\cap\Omega$ the same must be true for $L'$ and so $L'$ is contained in $\Omega$. If $\beta_2$ is a point in $I'$ then a similar argument as above shows that the line from $\alpha_2$ to $\beta_2$ is contained in $\Omega$ so the result follows.
\end{proof}

\begin{lem}
Suppose $I$ and $I'$ are the relative interiors of two proper faces of $\Omega$ and there exists a rigid line $L$ from $I$ to $I'$. If $L'$ is any line from $I$ to $I'$ then $L'$ is rigid.
\end{lem}
\begin{proof}
Let $\alpha$ and $\beta$ be the points in $I$ and $I'$ respectively which are the accumulation points of $L$ in $\partial\Omega$. Let $\alpha'$ be a point in $I$ and $\beta'$ be a point in $I'$. If the line from $\alpha'$ to $\beta'$ is not rigid then proposition 3.3 implies that there exists an open segment $S_\alpha' \subset I$ containing $\alpha'$ and an open segment $S_\beta'\subset I'$ containing $\beta'$ such that $S_\alpha'$ and $S_\beta'$ span a 2-dimensional plane.

Since $\alpha$ and $\alpha'$ lie in the interior of the open convex set $I$ there is an open segment $S_\alpha\subset I$ containing $\alpha$ which points in the same direction as $S_\alpha'$. Similarly we can find an open segment $S_\beta\subset I'$ containing $\beta$ which points in the same direction as $S_\beta'$. Hence $S_\alpha$ and $S_\beta$ span a 2-dimensional plane so this contradicts the assumption that $L$ is rigid.
\end{proof}

If $I$ and $I'$ are opposite in $\Omega$ and all the lines between them are rigid we will call them \textit{opposite rigid faces}. The \textit{join} $J$ of $I$ and $I'$ is the set of all lines from $I$ to $I'$. The join $J$ is \textit{rigid} if $I$ and $I'$ are opposite rigid faces of $\Omega$.

\begin{lem}
Suppose $I$ and $I'$ are the relative interiors of two proper faces of $\Omega$ which are opposite then the join $J$ of $I$ and $I'$ is a convex subset of $\Omega$.
\end{lem}
\begin{proof}
Lemma 8.2 implies that $J\subset\Omega$ so all we need to show is that $J$ is convex. Let $x$ and $y$ be two points in $J$. There exists lines $L_x$ and $L_y$ from $I$ to $I'$ in $J$ which contain $x$ and $y$ respectively. Let $\alpha_x, \alpha_y$ be the accumulation points of $L_x$ and $L_y$ in $I$ and $\beta_x,\beta_y$ be the accumulation points of $L_x$ and $L_y$ in $I'$. By lemma 8.1 we know that $[\alpha_x,\alpha_y]$ and $[\beta_x,\beta_y]$ are contained in $I$ and $I'$ respectively. Thus the join of $[\alpha_x,\alpha_y]$ and $[\beta_x,\beta_y]$ is contained in $J$.

We can see that the join of $[\alpha_x,\alpha_y]$ and $[\beta_x,\beta_y]$ is either a 2 or 3 dimensional convex set and therefore $[x,y]$ is in $J$.
\end{proof}

The next lemma will be used to prove that if two rays in $\Omega$ converge to points in $\partial\Omega$ which lie in an open line segment in $\partial\Omega$ then we can find sequences of points on each of the rays whose Hilbert distance remains bounded as they approach $\partial\Omega$.

\begin{lem}
Suppose $r_1$ and $r_2$ are rays in $\Omega$ with the following properties:

\vspace{3mm}
(i) There exists an open line segment S in $\partial\Omega$ such that $a(r_1)$ and $a(r_2)$ lie \indent in $S$.

(ii) $r_1$ and $r_2$ are contained in a 2-dimensional plane.

\vspace{3mm}
\noindent Then there exists sequences of points $\{x_n\}$ on $r_1$ and $\{y_n\}$ on $r_2$ converging in $\mathbb{R}^n$ to $a(r_1)$ and $a(r_2)$ respectively and a constant $C$ such that $d_\Omega(x_n,y_n)\leq C$ for all $n$.
\end{lem}
\begin{proof}
Note that condition $(i)$ means $a(r_1)$ and $a(r_2)$ lie in the relative interior of the same face of $\Omega$ by lemma 8.1. Since $r_1$ and $r_1$ are contained in a 2-dimensional plane lemma 7.3 allows us to assume that $\Omega$ lies in $\mathbb{R}^2$, see figure 6.

If $S\subset\partial\Omega$ is the open line segment containing $a(r_1)$ and $a(r_2)$ then we can choose sequences $\{x_n\}$ on $r_1$ and $\{y_n\}$ on $r_2$ converging in $\mathbb{R}^n$ to $a(r_1)$ and $a(r_2)$ respectively so that the straight line between $x_n$ and $y_n$ is parallel to $S$. Let $[\alpha,\beta]$ be the maximal line segment in $\partial\Omega$ containing $S$ and $\alpha_n$ and $\beta_n$ be the points where the line between $x_n$ and $y_n$ intersects $\partial\Omega$. Then $\alpha_n$ converges to $\alpha$ and $\beta_n$ converges to $\beta$. If follows that
\[\lim_{n\rightarrow\infty}d_\Omega(x_n,y_n)=ln(CR(\alpha,a(r_1),a(r_2),\beta))\]

\end{proof}

\vspace{5mm}
\begin{figure}[h]
    \centering
    \def\svgwidth{\columnwidth}
    \begin{overpic}[]{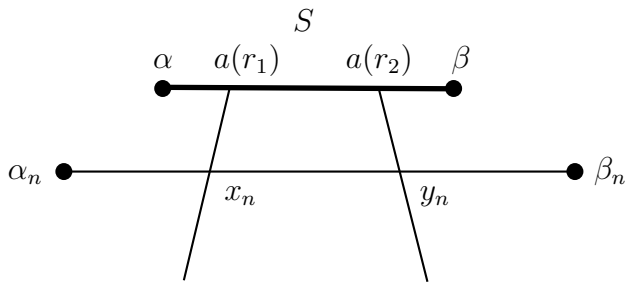}
        \put (45,48) {$S$}
        \put (18.5,41) {$\alpha$}
        \put (75,41) {$\beta$}
        \put (30,41) {$a(r_1)$}
        \put (55,41) {$a(r_2)$}
        \put (-9,20) {$\alpha_n$}
        \put (102,20) {$\beta_n$}
        \put (32,16) {$x_n$}
        \put (69,16) {$y_n$}

    \end{overpic}
    \vspace{3mm}
    \caption{$d_\Omega(x_n,y_n)$ remains bounded as $x_n$ and $y_n$ approach $\partial\Omega$.}
\end{figure}

The next lemma we will need tells us that if two rays in $\Omega$ converge to the same point in the boundary of $\Omega$ then we can find sequences of points on each of the rays whose Hilbert distance remains bounded as they approach $\partial\Omega$.

\begin{lem}
Suppose $r_1$ and $r_2$ are rays in $\Omega$ with $\alpha=a(r_1)=a(r_2)$. Then there exists sequences of points $\{x_n\}$ on $r_1$ and $\{y_n\}$ on $r_2$ converging in $\mathbb{R}^n$ to $\alpha$ and a constant $C$ such that $d_\Omega(x_n,y_n)\leq C$ for all $n$.
\end{lem}
\begin{proof}

By lemma 7.3 we can reduce to the case where $\Omega$ is a subset of $\mathbb{R}^2$ by intersecting $\Omega$ with the two dimensional plane containing $r_1$ and $r_2$. If $\alpha$ lies in the relative interior of a 1-dimensional face of $\Omega$ then we are done by lemma 8.5.

Otherwise $\alpha$ is an extreme point of $\Omega$. We can find a line $L_1$ with accumulation point $\alpha$ in $\Omega$ such that $r_1$ and $r_2$ are contained in the same open half space $H_1$ determined by $L_1$. And we can find a second line $L_2$ with accumulation point $\alpha$ in $\Omega\cap H_1$ such that $r_1, r_2$ and $L_1$ are contained in the same open half space $H_2$ determined by $L_2$, see figure 7.

Since the intersection of convex sets is convex $\Omega'=\Omega\cap H_1\cap H_2$ is a convex domain containing $r_1$ and $r_2$ which is contained in $\Omega$. Furthermore $\alpha$ is an extreme point of $\Omega'$ and $L_1$ and $L_2$ are contained in $\partial\Omega'$. Since $\Omega'\subset\Omega$ is follow from theorem 2.2 in ~\cite{HilbertKlein} that $d_\Omega\leq d_{\Omega'}$ on $\Omega'$.

Suppose $\{x_n\}$ is a sequence on $r_1$ converging to $\alpha$, then we can choose a sequence $\{y_n\}$ on $r_2$ converging to $\alpha$ such that the points where the line $L_n$ between $x_n$ and $y_n$ intersect $\partial\Omega'$ always lie on $L_1$ and $L_2$. Let $\alpha_n$ and $\beta_n$ be the points of $L_n$ on $L_1$ and $L_2$ respectively, see figure 7. By the cross ratio theorem $CR(\alpha_n,x_n,y_n,\beta_n)$ is constant for all $n$. Hence,
\[d_\Omega(x_n,y_n)\leq d_{\Omega'}(x_n,y_n)=C\]

\noindent for all $n$.

\end{proof}

\begin{figure}[h]
    \centering
    \def\svgwidth{\columnwidth}
    \begin{overpic}[]{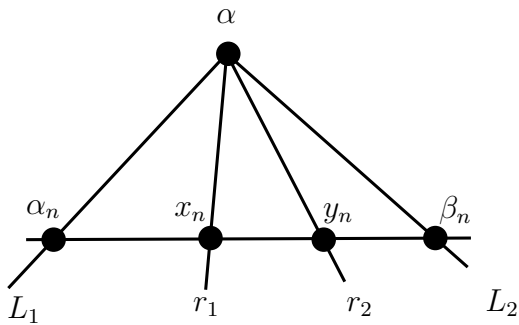}
        \put (45,58) {$\alpha$}
        \put (0,-5.5)  {$L_1$}
        \put (40,-4) {$r_1$}
        \put (73,-4) {$r_2$}
        \put (103,-4) {$L_2$}
        \put (4,17) {$\alpha_n$}
        \put (93,16) {$\beta_n$}
        \put (36,16.5) {$x_n$}
        \put (68,16.25) {$y_n$}
    \end{overpic}
    \vspace{3mm}
    \caption{$d_\Omega(x_n,y_n)$ remains bounded as $x_n$ and $y_n$ approach $\partial\Omega$.}
\end{figure}

\begin{cor}
Suppose $r_1$ and $r_2$ are rays in $\Omega$ such that $a(r_1)$ and $a(r_2)$ are contained in the relative interior $I$ of the same proper face of $\Omega$. Then there exists sequences of points $\{x_n\}$ on $r_1$ and $\{y_n\}$ on $r_2$ converging in $\mathbb{R}^n$ to $a(r_1)$ and $a(r_2)$ respectively and a constant $C$ such that $d_\Omega(x_n,y_n)\leq C$ for all $n$.
\end{cor}
\begin{proof}
If $a(r_1)=a(r_2)$ then we are done by lemma 8.6, so suppose $a(r_1)\neq a(r_2)$. Lemma 8.1 implies there exists an open segment $S\subset\partial\Omega$ containing $a(r_1)$ and $a(r_2)$.

Let $L_1$ be the line in $\Omega$ containing $r_1$ and define $\beta$ to be the second accumulation point of $L_1$ in $\partial\Omega$. By lemma 8.2 the line $L'$ from $\beta$ to $a(r_2)$ is contained in $\Omega$ and so we can find a ray $r'$ on $L'$ with $a(r')=a(r_2)$. By lemma 8.5 there exists sequences of points $\{x_n\}$ on $r_1$ and $\{z_n\}$ on $r'$ converging to $a(r_1)$ and $a(r_2)$ respectively and a constant $C$ such that $d_\Omega(x_n,z_n)\leq C$ for all $n$.

By using the idea in the proof of lemma 8.5 there exists a sequence of points $\{y_n\}$ on $r_2$ converging to $a(r_2)$ and a constant $C'$ such that $d_\Omega(y_n,z_n)\leq C'$ for all $n$. Thus $d_\Omega(x_n,y_n)\leq C+C'$ for all $n$.
\end{proof}

If two rays in $\Omega$ converge to points in $\partial\Omega$ which are not contained in a common face of $\Omega$ then the limit of any two sequences of points on these rays converging to the boundary points will be infinite, see theorem 5.2 of ~\cite{GromovHyperbolic}. The following lemma tells us the same is true if one ray converges to an extreme point of $\Omega$ and the other converges to any other point in $\partial\Omega$. This lemma will help us understand how families of lines converging to the same extreme point of $\Omega$ are mapped by a Hilbert isometry.

\begin{lem}
Suppose $r_1$ and $r_2$ are rays in $\Omega$ with $a(r_1)\neq a(r_2)$. Furthermore let $a(r_1)$ be an extreme point of $\Omega$. If $\{x_n\}$ is any sequence of points on $r_1$ converging to $a(r_1)$ and $\{y_n\}$ is any sequence of points on $r_2$ converging to $a(r_2)$, then
\[\lim_{n\rightarrow\infty}d_\Omega(x_n,y_n)=\infty\]
\end{lem}
\begin{proof}
Let $\{x_n\}$ be a sequence of points on $r_1$ converging to $a(r_1)$ and $\{y_n\}$ be a sequence of points on $r_2$ converging to $a(r_2)$. We can assume that $x_n$ and $y_n$ are distinct so we can define $L_n$ to be the straight line which contains $x_n$ and $y_n$. $L_n$ intersects $\partial\Omega$ in two points $\alpha_n$ and $\beta_n$ where $|\alpha_n-x_n|<|\alpha_n-y_n|$ and $|\beta_n-y_n|<|\beta_n-x_n|$. Since $a(r_1)$ is an extreme point of $\Omega$ it follows that $\alpha_n$ converges to $a(r_1)$ or $\beta_n$ converges to $a(r_2)$. The definition of the Hilbert metric implies that
\[\lim_{n\rightarrow\infty}d_\Omega(x_n,y_n)=\infty\]

\end{proof}

We will now give several lemmas that describe how isometries affect the accumulation points of rays in $\Omega$. The case we are most concerned with is when we have two rays whose accumulation point in $\partial\Omega$ is the same extreme point of $\Omega$. We want to understand when the images of these two rays under an isometry still have the same accumulation point in the boundary of $\Omega'$.

The proof of lemma 8.9 will use the Gromov product which we now define. Fix a base point $b$ in $\Omega$ and let $x$ and $y$ be in $\Omega$. Then the Gromov product of $x$ and $y$ is given by
\[ (x,y)_b=\frac{1}{2}(d_{\Omega}(x,b)+d_\Omega(y,b)-d_\Omega(x,y)).\]

\begin{lem}
Suppose $r_1$ and $r_2$ are rigid rays in $\Omega$ and $a(r_1)$ and $a(r_2)$ are contained in the relative interior $I$ of the same face of $\Omega$. If $\tau$ is an isometry of $\Omega$ then $[a(\tau(r_1)),a(\tau(r_2))]\subset\partial\Omega$.
\end{lem}
\begin{proof}
Corollary 8.7 implies there exists sequences of points $\{x_n\}$ on $r_1$ and $\{y_n\}$ on $r_2$ converging to $a(r_1)$ and $a(r_2)$ respectively and a constant $C$ such that $d_\Omega(x_n,y_n)\leq C$ for all $n$. From the definition of the Gromov product we can see that
\[(x_n,y_n)_b\rightarrow\infty\]

\noindent as $n\rightarrow\infty$. If $[a(\tau(r_1)),a(\tau(r_2))]\not\subset\partial\Omega$ then theorem 5.2 of ~\cite{GromovHyperbolic} implies
\[\limsup_{n\rightarrow\infty}(\tau(x_n) | \tau(y_n))_{\tau(b)}\leq K \]

\noindent for some constant $K$. This contradicts the fact that $\tau$ is an isometry and so the conclusion follows.
\end{proof}

\begin{cor}
Suppose $r_1$ and $r_2$ are rigid rays in $\Omega$ and $a(r_1)$ and $a(r_2)$ are contained in the relative interior $I$ of the same face of $\Omega$. If $\tau:\Omega\rightarrow\Omega'$ is an isometry then $a(\tau(r_1))$ and $a(\tau(r_2))$ lie in the relative interior of the same face of $\Omega'$.
\end{cor}
\begin{proof}
If $a(\tau(r_1))=a(\tau(r_2))$ then the result is trivial so will assume these are two distinct points.

Using corollary 8.7 we can find a sequence $\{x_n\}$ on $r_1$ converging to $a(r_1)$ and a sequence $\{y_n\}$ on $r_2$ converging to $a(r_2)$ and a constant $C$ such that
\[ d_{\Omega'}(\tau(x_n),\tau(y_n))=d_\Omega(x_n,y_n)\leq C \]
\noindent for all $n$.

Let $P'$ be a two dimensional plane containing $a(\tau(r_1))$ and $\tau(r_2)$, then $D'=P'\cap\Omega'$ is a two dimensional cross section of $\Omega'$ whose Hilbert metric agrees with the metric inherited from $d_{\Omega'}$, see lemma 7.3. Let $r_3$ be the ray in $D'$ which begins at the same point as $\tau(r_2)$ and converges to $a(\tau(r_1))$, see figure 8. From corollary 8.7 we can find a third sequence of points $\{z_n\}$ on $r_3$ converging to $a(\tau(r_1))$ and a constant $C'$ such that
\[ d_{\Omega'}(\tau(x_n),z_n)\leq C' \]

\noindent for all $n$. It follows from the triangle inequality that
\[ d_{D'}(\tau(y_n),z_n)=d_{\Omega'}(\tau(y_n),z_n)\leq C + C' \]

\noindent for all $n$. Because of the previous inequality lemma 8.8 implies that $a(\tau(r_1))$ and $a(\tau(r_2))$ are not extreme points of $D'$. But lemma 8.9 implies that $[a(\tau(r_1)),a(\tau(r_2))]$ is a subset of the relative boundary of $D'$ and thus $a(\tau(r_1))$ and $a(\tau(r_2))$ are contained in the relative interior of the same 1-dimensional face of $D'$. So we can find an open interval $S$ in the relative boundary of $D'$ containing $a(\tau(r_1))$ and $a(\tau(r_2))$. But $S$ must also be contained in the boundary of $\Omega'$ because $D'$ is a cross section of $\Omega'$ so the result follows from lemma 8.1.

\end{proof}

\begin{figure}[h]
    \centering
    \def\svgwidth{\columnwidth}
    \begin{overpic}[]{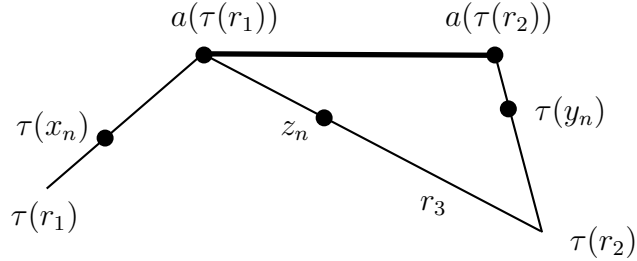}
        \put (-7,2) {$\tau(r_1)$}
        \put (105,-3) {$\tau(r_2)$}
        \put (-6,20) {$\tau(x_n)$}
        \put (98,23) {$\tau(y_n)$}
        \put (47,20) {$z_n$}
        \put (75,5) {$r_3$}
        \put (25,42) {$a(\tau(r_1))$}
        \put (80,42) {$a(\tau(r_2))$}
    \end{overpic}
    \vspace{3mm}
    \caption{The points $a(\tau(r_1))$ and $a(\tau(r_2)$ must lie in the relative interior of the same face of $\Omega'$.}
\end{figure}

We can now prove the main theorem of this section which tells us that if $J$ is the join of the relative interiors of two opposite rigid faces of $\Omega$ then $\tau(J)$ is the join of the relative interiors of two opposite rigid faces of $\Omega'$.

\begin{thm}
Suppose $\tau:\Omega\rightarrow\Omega'$ is an isometry. Let $I_1$ and $I_2$ be the relative interiors of two opposite rigid faces of $\Omega$ and let $J$ be the join of $I_1$ and $I_2$. If $L$ is a line from $I_1$ to $I_2$ then define $I_1'$ and $I_2'$ to be the opposite rigid faces of $\Omega'$ which contain the accumulation points of $\tau(L)$ with the same orientation. If $J'$ is the join of $I_1'$ and $I_2'$ then $J'$ is rigid and $\tau(J)=J'$.
\end{thm}
\begin{proof}
We first mention that $I_1'$ and $I_2'$ are opposite rigid faces of $\Omega'$ because of lemma 8.3 and the fact isometries map rigid lines to rigid lines. So by definition $J'$ is rigid.

If $M$ is a line from $I_1$ to $I_2$ then corollary 8.10 implies that $\tau(M)$ is a rigid line from $I_1'$ to $I_2'$ and thus $\tau(J)\subset J'$. By symmetry we get the opposite inclusion which completes the proof.
\end{proof}



\section{Rigid Cross Sections}
If $\tau$ is an isometry of Hilbert domains then it maps a rigid line to another rigid line. In this section we will generalize this idea to higher dimensions by defining a rigid cross section. This will allow us to use induction and apply ideas in 2 and 3 dimensions to domains of dimension greater than 3.

\begin{defn}
Suppose $\tau:\Omega\rightarrow\Omega'$ is an isometry and $D$ is a cross section of $\Omega$. If $\tau(D)$ is a cross section of $\Omega'$ then $D$ is \textit{rigid under $\tau$}.
\end{defn}

\begin{lem}
If $D$ is an m-dimensional cross section of $\Omega$ that is rigid under $\tau$ then $\tau(D)$ is an m-dimensional cross section of $\Omega'$.
\end{lem}
\begin{proof}
By definition $\tau(D)$ is a cross section of $\Omega'$. It follows from invariance of domain that $dim(\tau(D))=m$.
\end{proof}

\begin{lem}
If $\tau:\Omega\rightarrow\Omega'$ is an isometry and $D$ is a cross section of $\Omega$ that is rigid under $\tau$ then $\tau:(D,d_D)\rightarrow(\tau(D),d_{\tau(D)})$ is a Hilbert isometry.
\end{lem}
\begin{proof}
By lemma 7.3 and lemma 9.2 the metrics $d_D$ and $d_\Omega$ agree on $D$ and $d_{\tau(D)}$ and $d_{\Omega'}$ agree on $\tau(D)$. It follows immediately that $\tau|_D$ is a Hilbert isometry.
\end{proof}

\begin{lem}
If $D$ is a 2-dimensional cross section of $\Omega$ and the following properties are satisfied:
\begin{enumerate}
\item $\partial D$ contains two extreme points $e_1, e_2$ of $\Omega$.

\item $\partial D$ contains a third point $\alpha$ which is not contained in $[e_1,e_2]$ and every straight line in $D$ from $\alpha$ to a point on $(e_1,e_2)$ is a rigid line in $\Omega$.

\end{enumerate}
Then $D$ is rigid under any isometry $\tau:\Omega\rightarrow\Omega'$.
\end{lem}
\begin{proof}
Let $e_1, e_2, \alpha$ be the three points in $\partial D$ that satisfy condition 1 and 2 in the previous definition. Because $e_1$ and $e_2$ are extreme points of $\Omega$ and $\alpha$ does not lie in $[e_1,e_2]$ it follows that $\alpha$ can not lie on the straight line containing $e_1$ and $e_2$. Let $\Delta\subset D$ be the interior of the triangle with vertices $e_1, e_2, \alpha$.

Pick a point $x$ in $\Delta$ and define $L_i$ to be the straight line in $\Delta$ which contains $(x, e_i)$. We can choose a straight line $L_\alpha$ in $D$ that goes from $\alpha$ to a point in $(e_1,e_2)$ and intersects $L_i$ in a point $\alpha_i$ which is different from $x$, see figure 9. Define $P'$ to be the two dimensional plane that contains the three points $\tau(\alpha_1), \tau(\alpha_2), \tau(x)$ and let $D'=\Omega'\cap P'$.

We will show that $\tau(D)=D'$. The line $L_\alpha$ is rigid by assumption, thus $\tau(L_i)$ and $\tau(L_\alpha)$ are straight lines in $\Omega'$. Each of these three lines contain two points in $D'$ so the entire line must be contained in $D'$. Any point in $\Delta$ lies on a rigid line in $\Omega$ which goes from one of the vertices of $\Delta$ to the opposite edge of $\Delta$ and passes through two of the three lines $L_i$ and $L_\alpha$ in distinct points. Since the image of a rigid line under $\tau$ is a rigid line it follows that $\tau(\Delta)\subset D'$. If $y$ lies in $D$ but outside of $\Delta$ the it lies on a rigid line which goes from a vertex of $\Delta$ and passes through two points in $D$ whose images under $\tau$ are contained in $D'$. The image of this line under $\tau$ is contained in $D'$ so $\tau(D)\subset D'$.

Because $\tau$ is an isometry $\tau(D)$ is an open and closed subset of the connected set $D'$ so $\tau(D)=D'$.
\end{proof}

\begin{figure}[h]
    \centering
    \def\svgwidth{\columnwidth}
    \begin{overpic}[]{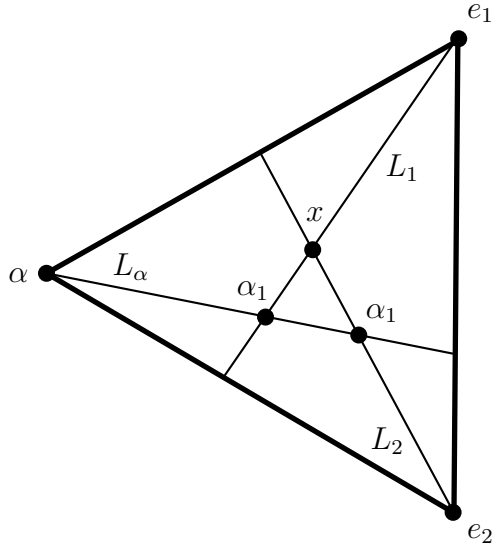}
        \put (-6,48.75) {$\alpha$}
        \put (86.5,102) {$e_1$}
        \put (86.5,-3) {$e_2$}
        \put (70,70) {$L_1$}
        \put (67,15) {$L_2$}
        \put (15, 50) {$L_\alpha$}
        \put (54, 61) {$x$}
        \put (40, 46) {$\alpha_1$}
        \put (66, 41.5) {$\alpha_1$}
    \end{overpic}
    \vspace{3mm}
    \caption{Lemma 9.4}
\end{figure}

\begin{cor}
If $D$ is a 2-dimensional cross section of $\Omega$ which contains an extreme triangle of $\Omega$ and $\tau:\Omega\rightarrow\Omega'$ is an isometry then $D$ is rigid under $\tau$.
\end{cor}
\begin{proof}
Because $D$ contains an extreme triangle the conditions of lemma 9.4 are satisfied so $D$ is rigid under $\tau$.
\end{proof}

\section{Isometries with Focusing Points}

If $\tau:\Omega\rightarrow\Omega'$ is a projective transformation of Hilbert domains then it maps any two rays in $\Omega$ converging to the same point in $\partial\Omega$ to two rays in $\Omega'$ which converge to the same point in $\partial\Omega'$. This property does not hold for the quadratic transformation of the triangle as discussed in section 4 and for the isometries of non Lorentzian symmetric cones which are not projective. Understanding this property in more generality is key to proving that a given isometry is a projective transformation and leads to the following definition.

\begin{defn}
Suppose $\tau:\Omega\rightarrow\Omega'$ is an isometry. An extreme point $e$ of $\Omega$ is a {\em focusing point} of $\tau$ if for any two rays $r_1, r_2$ in $\Omega$ with $e=a(r_1)=a(r_2)$ we have $f=a(\tau(r_1))=a(\tau(r_2))$. We will say that {\em $e$ focuses to $f$ under $\tau$}. A {\em focusing m-simplex} of $\tau$ is an extreme $m$-simplex in $\Omega$ whose vertices are focusing points of $\tau$.
\end{defn}

Every extreme point of $\Omega$ is a focusing point of $\tau$ if $\tau$ is a projective transformation. In this thesis we will show focusing points characterize the isometries of the Hilbert metric by showing that any isometry with a focusing point is a projective transformation. The next lemma and corollary show that a focusing point of $\tau$ must focus to an extreme point of $\Omega'$.

\begin{lem}
Suppose $\tau:\Omega\rightarrow\Omega'$ is an isometry and $e$ is an extreme point of $\Omega$. There exists a ray $r_1$ with $a(r_1)=e$ such that $a(\tau(r_1))$ is an extreme point of $\Omega'$ if and only if $e$ is a focusing point of $\tau$.
\end{lem}
\begin{proof}
Suppose there exists a ray $r_1$ in $\Omega$ with $a(r_1)=e$ such that $a(\tau(r_1))$ is an extreme point of $\Omega'$. If $r_2$ is a ray in $\Omega$ with $a(r_2)=e$ then $r_2$ is rigid so it follows from corollary 8.10 that $a(\tau(r_1))=a(\tau(r_2))$ and so $e$ is a focusing point of $\tau$.

For the other direction suppose that $e$ is a focusing point of $\tau$. Suppose $r_1$ is a rigid ray with $a(r_1)=e$ but $a(\tau(r_1))$ is not an extreme point of $\Omega'$. Then there is an open line segment $S\subset\partial\Omega'$ that contains $a(\tau(r_1))$. If $L$ is the straight line in $\Omega$ that contains $r_1$ then $\overline{\tau(L)}$ intersects $\partial\Omega$ at a point $\alpha_1$ in $S$ and at a second point $\beta$ which is not in $S$, see figure 10. Pick a point $\alpha_2$ in $S$ which is distinct from $\alpha_1$. Since $\tau(L)$ is a rigid line it follows from lemma 8.3 that the straight line $(\beta, \alpha_2)$ is a rigid line in $\Omega'$. By corollary 8.10, $\tau^{-1}((\beta, \alpha_2))$ is rigid line in $\Omega$ whose closure contains $e$. This is a contradiction because $e$ is a focusing point and $\alpha_1$ and $\alpha_2$ were chosen to be distinct points in $\partial\Omega'$.
\end{proof}

\vspace{10mm}
\begin{figure}[h]
    \centering
    \def\svgwidth{\columnwidth}
    \begin{overpic}[]{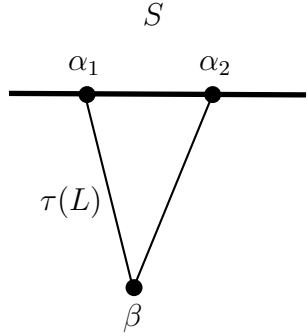}
        \put (20,77) {$\alpha_1$}
        \put (64,77) {$\alpha_2$}
        \put (38.5,-9) {$\beta$}
        \put (10.5,30) {$\tau(L)$}
        \put (45,92) {$S$}
    \end{overpic}
    \vspace{3mm}
    \caption{Lemma 10.2}
\end{figure}

\begin{cor}
If $e$ is a focusing point of the isometry $\tau:\Omega\rightarrow\Omega'$ and $e$ focuses to the point $f$ under $\tau$ then $f$ is an extreme point of $\Omega'$.
\end{cor}

If an isometry has a focusing point then the family of lines through the focusing point in the domain is mapped to a family of lines through a focusing point in the codomain with the same orientation. If a continuous map of $\mathbb{R}^n$ maps $0$ to $0$ and maps $n+1$ families of lines in general position to these same $n+1$ families of lines then it is the identity of the projective linear group, see lemma 10.4. We will use this idea to show that if an isometry has $n+1$ focusing points in general position then it is a projective transformation. We have already seen this idea in section 4 where the isometry $\Lambda_2$ maps the families of lines through the vertices of the triangle to three families of lines in the hexagonal plane.

\begin{lem}
Suppose $n\geq 2$ and $f:\mathbb{R}^n\rightarrow\mathbb{R}^n$ is a continuous function satisfying the following:

\vspace{5mm}
1.) f(0)=0

2.) If $L$ is a line in $\mathbb{R}^n$ parallel to the $i^{th}$ coordinate axis of $\mathbb{R}^n$ then $f(L)$ is \indent a line in $\mathbb{R}^n$ parallel to the $i^{th}$ coordinate axis of $\mathbb{R}^n$.

3.) f sends lines parallel to the line $t\rightarrow(t,...,t)$ to lines parallel to the line \indent $t\rightarrow(t,...,t)$.

\vspace{5mm}
\noindent then there exists $k\in\mathbb{R}\setminus\{0\}$ such that $f(\textbf{x})=k\textbf{x}$ for all $\textbf{x}\in\mathbb{R}^n$.
\end{lem}
\begin{proof}
The function $f$ has the following form
\[ f(x_1,...,x_n)=(f_1(x_1,...,x_n),...,f_n(x_1,...,x_n)) \]

\noindent Condition 2 of the lemma implies that the coordinate function $f_i$ of $f$ only depends on the variable $x_i$. Choose $j\in\{1,...,n\}$ such that $j\neq i$. We will show that $f_i$ does not depend on the variable $x_j$. Let $\textbf{x}$ and $\textbf{x}'$ be two points in $\mathbb{R}^n$ such that all of the coordinates of $\textbf{x}$ and $\textbf{x}'$ are the same except for the $j^{th}$ coordinate. It follows that $\textbf{x}$ and $\textbf{x}'$ lie on a line parallel to the $j^{th}$ coordinate axis of $\mathbb{R}^n$ and therefore the same is true for $f(\textbf{x})$ and $f(\textbf{x}')$. Thus $f_i(\textbf{x})=f_i(\textbf{x}')$ and we are done.

Using condition 1 of the lemma we see that $f_i(0)=0$ for all $i$. Consider the line $L_a$ defined by the equation
\[x(t)=(a,a,0,...,0)+t(1,1,...,1)\]

\noindent where $a$ is any real number. Since $L_a$ is parallel to the line $t\rightarrow(t,...,t)$ condition 3 of the lemma implies that $f(L_a)$ is a line parallel to the line $t\rightarrow(t,...,t)$. Thus $f_1(x(t))-f_2(x(t))=c$ is a constant that doesn't depend on $t$. Plugging in $t=-a$ we see that $c=0$ because $f_i(0)=0$ and $f_i$ only depends on $x_i$. Thus $f_1(x(t))=f_2(x(t))$ and so $f_1(a,a,0,...,0)=f_2(a,a,0,...,0)$ for all real numbers $a$. Because $f_1$ only depends on $x_1$ and $f_2$ only depends on $x_2$ we have $f_2(x_1,x_2,...,x_n)=f_1(x_2,0,...,0)$. The above argument does not depend on the coordinate so we have
\[ f(x_1,...,x_n)=(f_1(x_1,0,...,0),f_1(x_2,0,...0),...,f_1(x_n,0,...,0)) \]

Define the function $g:\mathbb{R}\rightarrow\mathbb{R}$ by $g(x)=f_1(x,0,...,0)$. The next step is to show that $g$ is linear. Let $a,b$ be real numbers and let $M_a$ be the line in $\mathbb{R}^n$ defined by the equation
\[x(t)=(a,0,0,...,0)+t(1,1,...,1)\]

By condition 3 of the lemma, $f_1(t+a,0,...,0)-f_1(t,0,...,0)=g(t+a)-g(t)$ is a constant that doesn't depend on $t$. Choosing $t=0$ and $t=b$ we see that $g(a)-0=g(a+b)-g(b)$ and therefore $g$ is additive. Because $g$ is an additive continuous function of $\mathbb{R}$ it must be linear and the conclusion follows.
\end{proof}

\begin{lem}
Suppose $\tau:\Omega\rightarrow\Omega'$ is an isometry and $\Delta$ is a focusing n-simplex of $\tau$ where $n=dim(\Omega)=dim(\Omega')$. If $e_0,..,e_n$ are the vertices of $\Delta$ and $e_i$ focuses to the point $f_i$ under $\tau$ then $\tau(\Delta)$ is a focusing n-simplex of $\tau^{-1}$ with vertices $f_0,...,f_n$.
\end{lem}
\begin{proof}
Let $\Delta'$ be the interior of the convex hull of the points $f_0,...,f_n$. By corollary 10.3 the extreme point $f_i$ is a focusing point of $\tau^{-1}$. First we prove that $\Delta'$ is a focusing $n$-simplex of $\tau^{-1}$. Let $p$ be a point in $\Delta$. The image of the $n$-simplex $(p,e_0,...,e_{n-1})$ under $\tau$ is an $n$-simplex with vertices $\tau(p),f_0,...,f_{n-1}$ because the vertices $e_i$ are focusing points of $\tau$. To see this first note that the line segment $(p,e_0)$ is mapped by $\tau$ to the line segment $(\tau(p),f_0)$. Because $f_1$ is a focusing point of $\tau$ we can see that the image of the triangle $(p,e_0,e_1)$ is a triangle $(\tau(p),f_0,f_1)$. We can continue in this manner to show that the image of the $(i+1)$-simplex $(p,e_0,...,e_i)$ is the $(i+1)$-simplex $(p,f_0,...,f_i)$. It follows that the points $f_0,...,f_{n-1}$ are the vertices of an $(n-1)$-simplex. The point $f_n$ does not lie in the $(n-1)$-dimensional hyperplane containing $f_0,...f_{n-1}$ because the line containing $p$ and $e_n$ intersects the interior of the $n$-simplex $[p,e_0,...,e_{n-1}]$. It follows that $\Delta'$ is an open $n$-simplex.

Now we will show that $\tau(\Delta)=\Delta'$. Suppose that none of the points in $\Delta$ are mapped into $\Delta'$ by $\tau$. Let $H_i$  be the $(n-1)$-dimensional hyperplane which contains all of the vertices of $\Delta'$ except for $f_i$ and let $H^+_i$ be the open halfspace determined by $H_i$ which contains $\Delta'$. We can choose a point $p$ in $\Delta$ so that $\tau(p)$ is contained in a maximal number of the halfspaces $H^+_i$. There exists an open set $U$ containing $p$ such that $U\subset\Delta$ so we can also choose $p$ so that $\tau(p)$ does not lie in any of the hyperplanes $H_i$. Without loss of generality we can assume that $\tau(p)$ does not lie in $H^+_0$ and therefore the straight line from $\tau(p)$ to $f_0$ contains a point $q$ that lies in $H^+_0$. If $i\neq0$ and $\tau(p)\in H_i^+$ then $(f_0,\tau(p))$ is contained in $H_i^+$ because $f_0$ is in $\overline{H_i^+}$. Hence $q$ is contained in $H_i^+$ as well. This contradicts the maximality condition of $p$ because $\tau^{-1}(q)$ lies in $\Delta$. Hence at least one point of $\Delta$ is mapped into $\Delta'$ by $\tau$.

Let $p$ be a point in $\Delta$ with $\tau(p)$ in $\Delta'$ and take $U$ to be an open set containing $p$ and lying inside of $\Delta$. Since $\tau(p)$ is in $\Delta'$ we can find an open ball $B$ which contains $\tau(p)$ so that $B\subset\tau(U)\cap\Delta'$. Thus the entire open set $\tau^{-1}(B)$ is contained in $\Delta$. Because the $e_i$ are focusing points of $\tau$ any line in $\Delta$ from $e_i$ through a point in $\tau^{-1}(B)$ is mapped inside of $\Delta'$ by $\tau$. Thus we can keep extending the region inside of $\Delta$ which is mapped by $\tau$ inside of $\Delta'$ to get $\tau(\Delta)\subset\Delta'$. By similar reasoning we get $\tau^{-1}(\Delta')\subset\Delta$ and therefore $\tau(\Delta)=\Delta'$.

\end{proof}

The next lemma guarantees that an isometry which is projective on an open set in a convex domain must be projective on the entire domain.

\begin{lem}
Suppose $\tau:\Omega\rightarrow\Omega'$ is an isometry and there is an open set $U$ in $\Omega$ such that the restriction of $\tau$ to $U$ is given by a projective transformation, then $\tau$ is a projective transformation.
\end{lem}
\begin{proof}
Suppose $\tau$ restricted to $U$ is given by the projective transformation $p$. We will prove that $\tau$ must be equal to $p$ on $\Omega$. The map $p^{-1}\tau:\Omega\rightarrow p^{-1}(\Omega')$ is the identity map on $U$. We first show that $p^{-1}(\Omega')=\Omega$. Let $x_1$ be any point in the boundary of $p^{-1}(\Omega')$ and choose a point $x_2$ in the boundary of $p^{-1}(\Omega')$ so that the line $[x_1,x_2]$ passes through the open set $U$. The straight line through $x_1$ and $x_2$ passes through the boundary of $\Omega$ in two points $y_1$ and $y_2$. Fix a point $a$ in $U$ that lies on the line $[x_1,x_2]$ then for any point $b$ in $U$ lying on this same line we have $CR(x_1,a,b,x_2)=CR(y_1,a,b,y_2)$. If $x_1\neq y_1$ then there is at most finitely many choices for $b$ that satisfy this equation. Thus we must have $x_1=y_1$ and by the same reasoning $x_2=y_2$. It follows that $\partial p^{-1}(\Omega')$ is a subset of $\partial\Omega$. By similar reasoning we find that $\partial p^{-1}(\Omega')$ contains $\partial\Omega$ and therefore $p^{-1}(\Omega')=\Omega$.

We now show that $p^{-1}\tau$ is the identity map on $\Omega$. If $e$ is an extreme point of $\Omega$ then $p^{-1}\tau$ is the identity on any line through $e$ which passes through $U$. From the Krein-Milman theorem we can find an extreme $n$-simplex $\Delta$ of $\Omega$ such that $\Delta\cap U\neq\emptyset$. Since $\Delta$ contains an open set on which $p^{-1}\tau$ is the identity it follows that $p^{-1}\tau_{\overline{\Delta}\cap\Omega}$ is the identity. From here one can keep extending the region on which $p^{-1}\tau$ is the identity to all of $\Omega$ so $\tau=p$ on $\Omega$.
\end{proof}

In the proof of lemma 10.7 we will show that if an isometry has a focusing $n$-simplex then it is a projective transformation on this $n$-simplex. Using lemma 10.6 then implies that the isometry is a projective transformation on the entire domain.

In the proof of lemma 10.7 we will need to use the isometry $\Lambda_n:\Delta_n\rightarrow W_n$, see section 5 p. 21 and ~\cite{Harpe} for further details. The definition of $\Lambda_n$ is given by the following
\[\Lambda_n(x_1,...,x_{n+1})=(\theta_1,...,\theta_{n+1})\]

\noindent where
\[ \theta_i=\frac{1}{n+1}(nlog(x_i)-\sum_{j\neq i}log(x_j)). \]

\noindent The $i^{th}$  $axis$ of $W_n$ is the line in $W_n$ defined by $x_i=-nt$ and $x_j=t$ for $j\neq i$. One can check that the family of lines through a vertex of $\Delta_n$ is mapped by $\Lambda_n$ to lines parallel to one of the axis of $W_n$. For $1\leq i\leq n$ let $v_i=(s_1,...,s_{n+1})$ where $s_i=-n$ and $s_j=1$ for $j\neq i$. The vectors $v_i$ form a basis for $W_n$ so we can define a linear isomorphism $T_n:W_n\rightarrow\mathbb{R}^n$ by $T_n(v_i)=e_i$ where $e_i$ is the standard basis vector. One can check that $T_n$ maps lines parallel to the $i^{th}$  $axis$ of $W_n$ to lines parallel to the $i^{th}$  $axis$ of $\mathbb{R}^n$ for $1\leq i\leq n$. Furthermore $T_n$ maps lines parallel to the $(n+1)^{th}$  $axis$ of $W_n$ to lines parallel to the line defined by $t\rightarrow(t,...,t)$ in $\mathbb{R}^n$.

\begin{lem}
If $\tau:\Omega\rightarrow\Omega'$ is an isometry that has a focusing n-simplex $\Delta$, then $\tau$ is a projective transformation.
\end{lem}
\begin{proof}
If $\Omega$ is a $n$-simplex then $\Omega'$ must also be an $n$-simplex and so the result can be found in ~\cite{Polyhedral}. Thus we will assume that both $\Omega$ and $\Omega'$ are not $n$-simplicies.

By lemma 10.5 we have $\Delta'=\tau(\Delta)$ is a focusing $n$-simplex of $\tau^{-1}$. Denote the vertices of $\Delta$ by $e_i$ for $i=0,...,n$. We can choose projective transformations $p:\Delta_n\rightarrow\Delta$ and $p':\Delta'\rightarrow \Delta_n$ so that the map $p'\tau|_\Delta p:\Delta_n\rightarrow \Delta_n$ sends straight lines through a given vertex of $\Delta_n$ to lines through the same vertex of $\Delta_n$ with the orientation preserved. Thus the map $\Lambda_n p'\tau|_\Delta p\Lambda_n^{-1}:W_n\rightarrow W_n$ sends lines parallel to a given axis of $W_n$ to lines parallel to the same axis of $W_n$. Denote the metric on $\mathbb{R}^n$ induced by $T_n$ as $d_{T_n}$. Define $f$ to be the map $T_{n}\Lambda_n p'\tau p\Lambda^{-1}_nT_n^{-1}:\mathbb{R}^n\rightarrow \mathbb{R}^n$ composed with a translation so that $f(0)=0$. The map $f$ satisfies the conditions of lemma 10.4 so there exists a constant $k\in\mathbb{R}\setminus\{0\}$ such that $f(\textbf{x})=k\textbf{x}$ for all $\textbf{x}\in\mathbb{R}^n$.

We will now show that $k=1$. The map $f$ has the form $I_2\tau|_\Delta I_1$ where $I_2$ is an isometry from $(\Delta',d_{\Delta'})$ to $(\mathbb{R}^n,d_{T_n})$ and $I_1$ is an isometry from $(\mathbb{R}^n,d_{T_n})$ to $(\Delta,d_{\Delta})$. The metric $d_{T_n}$ on $\mathbb{R}^n$ is induced by a norm thus
\[d_{\Delta'}(\tau|_{\Delta}(x),\tau|_{\Delta}(y))=|k|d_{\Delta}(x,y)\]  for all $x,y\in\Delta$. Let $L$ be a straight line in $\Omega$ which intersects $\Delta$ and passes through the point $e_0$. Choose sequences of points $\{x_n\}$ and $\{y_n\}$, inside of $\Delta$ and lying on $L$, such that $|e_0-x_n|=\frac{1}{n}$ and $|e_0-y_n|=\frac{2}{n}$. From the definition of the Hilbert metric it follows that
\[ \lim_{n\rightarrow\infty}d_\Delta(x_n,y_n)=\lim_{n\rightarrow\infty}d_\Omega(x_n,y_n)=2 \]

\noindent and thus
\[ \lim_{n\rightarrow\infty}d_{\Delta'}(\tau(x_n),\tau(y_n))=\lim_{n\rightarrow\infty}d_{\Omega'}(\tau(x_n),\tau(y_n))=2 \]

\noindent because $\tau$ is an isometry from $\Omega$ to $\Omega'$ and $e_0$ is a focusing point of $\tau$. It follows from the equation above that $|k|=1$ and so $\tau:(\Delta,d_\Delta)\rightarrow(\Delta',d_{\Delta'})$ is an isometry. Because $\Delta$ is a focusing $n$-simplex $\tau|\Delta$ is a projective transformation and it follows from lemma 10.6 that $\tau$ is a projective transformation.
\end{proof}

\begin{lem}
If $\tau:\Omega\rightarrow\Omega'$ is an isometry and $L$ is an extreme line in $\Omega$ then $\tau(L)$ is an extreme line in $\Omega'$.
\end{lem}
\begin{proof}
An extreme line is rigid so $\tau(L)$ is a rigid line in $\Omega'$. Let $\alpha, \beta$ be the points where $\overline{\tau(L)}$ intersects $\partial\Omega'$. If $\alpha$ is not an extreme point of $\Omega'$ then there is an open line segment $S\subset\partial\Omega'$ containing $\alpha$. Choose $\alpha'$ in $S$ different from $\alpha$. The line $(\alpha', \beta)$ is rigid in $\Omega'$ by lemma 8.3. By corollary 8.10 we have $\tau^{-1}((\alpha',\beta))=L$. This is a contradiction because $\alpha'$ was chosen different from $\alpha$, thus $\alpha$ is an extreme point of $\Omega'$. Similarly $\beta$ is an extreme point of $\Omega'$ so $\tau(L)$ is an extreme line in $\Omega'$.
\end{proof}

The next theorem is the first main result of this thesis about isometries of the Hilbert metric. The existence of an extreme line in a convex domain means that all isometries of the domain are projective. Our theorem generalizes de la Harpe's theorem on strictly convex domains because strictly convex domains contain extreme lines. A large class of convex domains contain an extreme line. In particular in dimension 2 the only domain without an extreme line is a triangle, thus theorem 8.9 along with de la Harpe's study of the triangle finishes off the study of non projective isometries in this dimension. In dimension 3 the situation is already more complicated because any cone on a two dimensional convex domain does not contain an extreme line. Furthermore the join of a straight line and a rectangle is a 3 dimensional convex domain which does contain a extreme line, however this is a polyhedral domain so its isometries were studied in ~\cite{Polyhedral}.

\begin{thm}
Suppose $\tau:\Omega\rightarrow\Omega'$ is an isometry and $\Omega$ contains an extreme line, then $\tau$ is a projective transformation.
\end{thm}
\begin{proof}
Let $L$ be an extreme line in $\Omega$ which converges to extreme points $e_1$ and $e_2$ in the boundary of $\Omega$. From lemma 10.8 it follows that $e_1$ and $e_2$ are focusing points of $\tau$. We will use $f_1$ and $f_2$ to denote the extreme points of $\Omega'$ that $e_1$ and $e_2$ focus to under $\tau$. Suppose $e_3$ is another extreme point of $\Omega$. We will show that $e_3$ is a focusing point of $\tau$.

If at least one of the open line segments $(e_1,e_3)$ or $(e_2,e_3)$ is in $\Omega$ then it follows from lemma 10.2 and lemma 10.8 that $e_3$ is a focusing point of $\tau$ so we will assume that both $(e_1,e_3)$ and $(e_2,e_3)$ are contained in $\partial\Omega$. Let $r$ be a rigid ray that begins at some point $x$ on $L$ and converges to $e_3$. Because $(e_1,e_3)$ and $(e_2,e_3)$ are in $\partial\Omega$ it follows that $(f_1,a(\tau(r)))$ and $(f_2,a(\tau(r)))$ are in $\partial\Omega'$.

We will first consider the case when $\Omega$ and $\Omega'$ are 2 dimensional convex domains. Because we are in dimension 2 the extreme line $\tau(L)$ divides $\Omega'$ in to two halves. Since $(f_1,a(\tau(r)))$ and $(f_2,a(\tau(r)))$ are in $\partial\Omega'$ it follows immediately that $a(\tau(r))$ is an extreme point of $\Omega'$ so $e_3$ is a focusing point of $\tau$ by lemma 10.2.

Now assume that $\Omega$ and $\Omega'$ are convex domains of dimension greater than 2. Suppose that $a(\tau(r))$ is not an extreme point of $\Omega'$, then there exists an open line segment $S\subset\partial\Omega'$ which contains $a(\tau(r))$. If $P$ is the 2-dimensional plane containing $e_1,e_2$ and $e_3$ then D=$P\cap\Omega$ is a cross section of $\Omega$ that is rigid under $\tau$, see corollary 9.5. Hence $\tau(D)$ is a 2-dimensional cross section of $\Omega'$. From the 2-dimensional case we know that $a(\tau(r))$ is an extreme point of $\tau(D)$ so the segment $S$ only intersects $\partial\tau(D)$ in the single point $a(\tau(r))$.

Let $\alpha$ be a point in $S$ different from $a(\tau(r))$ and let $\Delta$ be the interior of the triangle with vertices $f_1,f_2,\alpha$, see figure 11. Note that $\alpha$ and $a(\tau(r))$ lie in the relative interior of the same face of $\Omega'$ by lemma 8.1. Let $P'$ be the 2-dimensional plane containing $\Delta$. The point $\alpha$ is an extreme point of $P'\cap\Omega'$ because $(f_1,\alpha)$ and $(f_2,\alpha)$ are contained in $\partial\Omega'$. Thus if $r'$ is a ray starting at $\tau(x)$ which converges to $\alpha$ then it is rigid in $P'\cap\Omega'$. Because $f_1$ and $f_2$ are focusing points of $\tau^{-1}$ it follows that $\tau^{-1}(\Delta)$ is contained in a 2-dimensional plane. Let $y$ be a point on $r'$. Then $[\tau(x),y]$ is the unique geodesic between $\tau(x)$ and $y$ in $P'\cap\Omega'$. It follows that $\tau^{-1}([\tau(x),y])=[x,\tau^{-1}(y)]$ because $\tau([x,\tau^{-1}(y)])$ is a geodesic in $P'\cap\Omega'$ from $\tau(x)$ to $y$ and therefore must be equal to $[\tau(x),y]$. By moving $y$ towards $\alpha$ we can see that $\tau^{-1}(r')$ is a ray in $\Omega$. From corollary 8.10 we find that $a(\tau^{-1}(r'))=e_3$ and thus $r=\tau^{-1}(r')$. This is a contradiction so $a(\tau(r))$ is an extreme point of $\Omega'$ and thus $e_3$ is a focusing point of $\tau$.

We have shown that any extreme point of $\Omega$ is a focusing point of $\tau$. Thus $\tau$ has a focusing $n$-simplex so it is a projective transformation from lemma 10.7.
\end{proof}


\begin{figure}[h]
    \centering
    \def\svgwidth{\columnwidth}
    \begin{overpic}[]{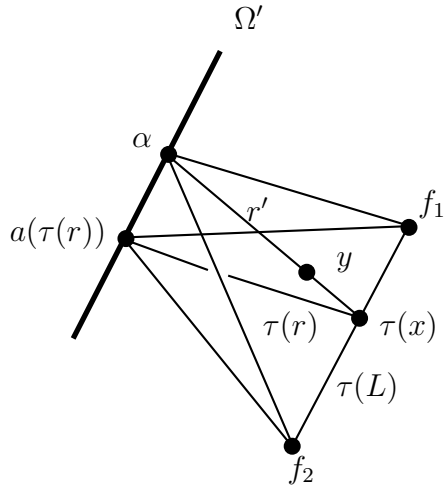}
        \put (40,105) {$\Omega'$}
        \put (-15,53) {$a(\tau(r))$}
        \put (15,76) {$\alpha$}
        \put (85,60) {$f_1$}
        \put (53,-5) {$f_2$}
        \put (75.5,30) {$\tau(x)$}
        \put (65,47) {$y$}
        \put (43,58) {$r'$}
        \put (65,15) {$\tau(L)$}
        \put (47,30) {$\tau(r)$}
    \end{overpic}
    \vspace{3mm}
    \caption{Theorem 10.9}
\end{figure}

\section{Isometries of the Hilbert Metric in $\mathbb{R}^2$}

In this section we will show that the only 2 dimensional convex domain that has an isometry which is not a projective transformations is the triangle. We have seen that in higher dimensions there are domains besides the $n$-simplex for which this is true. This makes the higher dimensional case more complicated than the 2 dimensional case. The main reason for this difference is that in two dimensions the only convex domain that does not contain an extreme line is the interior of a triangle whereas in higher dimensions there are more examples of such domains. We prove the following theorem which was stated at Theorem 1.4 in the introduction.

\begin{thm}
Suppose $\Omega$ and $\Omega'$ are the interiors of compact convex sets in $\mathbb{R}^2\subset\mathbb{R}P^2$ which are isometric. Then the following statements hold:

\begin{enumerate}

\item If $\Omega$ is not the interior of a triangle then $\Omega'$ is not the interior of a triangle and $Isom(\Omega, \Omega')=PGL(\Omega,\Omega')$.

\item (de la Harpe) If $\Omega$ is the interior of a triangle then $\Omega'$ is the interior of a triangle. Furthermore if $\Omega=\Omega'$ then $Isom(\Omega)\cong\mathbb{R}^2\rtimes D_6$ and $PGL(\Omega)\cong\mathbb{R}^2\rtimes D_3$ where $D_n$ is the dihedral group of order $2n$.

\item $\Omega$ and $\Omega'$ are isometric if and only if they are projectively equivalent.

\end{enumerate}

\end{thm}
\begin{proof}
Suppose $\tau:\Omega\rightarrow\Omega'$ is an isometry of open bounded convex sets in $\mathbb{R}^2$. If $\Omega$ is not the interior of a triangle then $\Omega$ contains an extreme line and the image of this line under $\tau$ is an extreme line in $\Omega'$ by lemma 10.8. Thus it follows that $\Omega'$ is not the interior of a triangle. Furthermore theorem 10.9 implies that $\tau$ is a projective transformation, so statement 1 holds.

If $\Omega$ is the interior of a triangle then it follow again from lemma 10.8 that $\Omega'$ is the interior of a triangle. If $\Omega=\Omega'$ it follows from proposition 4 and the corollary on pg 112 in ~\cite{Harpe} that $Isom(\Omega)\cong\mathbb{R}^2\rtimes D_6$ and $PGL(\Omega)\cong\mathbb{R}^2\rtimes D_3$, so statement 2 holds.

If $\tau:\Omega\rightarrow\Omega'$ is an isometry which is not projective then theorem 10.9 implies that $\Omega$ and $\Omega'$ do not contain extreme lines. Thus $\Omega$ and $\Omega'$ are triangles and are therefore projectively equivalent.

\end{proof}

\section{Minimal Cones}

In this section we will give the definition of minimal cone in a convex domain $\Omega$ and prove some properties about them which will be used to study the isometries of the Hilbert metric on general convex domains. Minimal cones are useful because they allow us to use induction when studying an isometry.

Suppose $e$ is an extreme point of a convex domain $\Omega$ and $I$ is the relative interior of a face of $\Omega$ which is opposite to $e$. Then the join of $e$ and $I$ is a cone with dimension of $dim(I)+1$ that is contained in $\Omega$ which we will denoted by $C_e(I)$. The point $e$ is a \textit{vertex} of $C_e(I)$ and $I$ is a \textit{base} of $C_e(I)$. The cone $C_e(I)$ is \textit{minimal} if its relative boundary is contained in the boundary of $\Omega$.

\begin{lem}
If $C_e(I)$ is a minimal cone in $\Omega$ then $C_e(I)$ is a cross section of $\Omega$ of dimension $dim(I)+1$. Furthermore if $\tau:\Omega\rightarrow\Omega'$ is an isometry then $C_e(I)$ is rigid under $\tau$.
\end{lem}
\begin{proof}
Define $m=dim(I)+1$. Let $P$ be the $m$-dimensional affine subspace in $\mathbb{R}^n$ which contains $C_e(I)$. Then we claim that $C_e(I)=P\cap\Omega$. Clearly we have $C_e(I)\subset P\cap\Omega$. So let $x$ be a point in $P\cap\Omega$ that is not in $C_e(I)$ and $y$ be a point in $C_e(I)$. Then $[x,y]$ is contained in $P\cap\Omega$ and so some point $z$ in $[x,y)$ is in the relative boundary of $C_e(I)$. Since $C_e(I)$ is a minimal cone it follows that $z$ is in $\partial\Omega$. This is a contradiction because $z$ is also in $P\cap\Omega$. We have shown that $C_e(I)$ is a cross section of $\Omega$ with dimension of $m$.

By theorem 8.11, $\tau(C_e(I))=J'$ where $J'$ is the join of the relative interiors $I_1'$ and $I_2'$ of two proper faces of $\Omega'$. Let $\alpha$ be a point in the relative boundary of $J'$ and $x$ be a point in $J'$. If $\alpha$ is in $\Omega'$ then $\tau^{-1}([\alpha,x])$ is in $\Omega$. But $(\alpha,x]$ is in $J'$ so $\tau^{-1}((\alpha,x])$ is in $C_e(I)$. Hence $\tau^{-1}(\alpha)$ is an accumulation point of $C_e(I)$ and is therefore in $\partial C_e(I)$. This contradicts the fact that $C_e(I)$ is minimal. It follows that the relative boundary of $J'$ is contained in $\partial\Omega'$. Using a similar argument as before this means that $J'$ is a cross section of $\Omega'$ with dimension $m$.
\end{proof}

\begin{lem}
If $e$ is an extreme point of a convex domain $\Omega$ then $\Omega$ contains a minimal cone with vertex $e$.
\end{lem}
\begin{proof}
Let $F_I$ be a face of $\Omega$ whose relative interior $I$ is opposite to $e$. If $C_e(I)$ is minimal then we are done.

If not then some point in the relative boundary of $C_e(I)$ is contained in $\Omega$. Let $L$ be the line through this point and $e$. The line $L$ intersects $F_I$ in a single point $\alpha$ which does not lie in $I$. Let $F'$ be the face in $F_I$ whose relative interior contains $\alpha$. It follows from Theorem 5.2 and Corollary 5.5 in ~\cite{Convex} that $F'$ is a face of $\Omega$ with $dim(F') < dim(F_I)$. The relative interior of $F'$ is the same in $F_I$ and $\Omega$ so if $I'$ is the relative interior of $F'$ then $C_e(I')$ is contained in $\Omega$. Since we have lowered the dimension eventually we will get a minimal cone with vertex $e$.
\end{proof}

The next lemma guarantees that any convex domain which is not a $n$-simplex has a minimal cone of dimension less than the dimension of the domain. We will use induction and these minimal cones of smaller dimension to understand how families of lines through an extreme point are mapped by an isometry of the Hilbert metric.

\begin{lem}
Suppose $\Omega$ is a convex domain that contains no minimal cone of dimension less than $n$, then $\Omega$ is an n-simplex.
\end{lem}
\begin{proof}
If $\Omega$ is not a cone then the lemma follows immediately because every minimal cone in $\Omega$ has dimension less than $n$. Thus we will assume that $\Omega$ is a cone.

The proof will use induction on the dimension of $\Omega$. If $n=2$ then the result is trivial because the only 2-dimensional convex cone is a triangle. Thus we will assume that $n>2$ and the statement of the lemma holds for dimension $n-1$.

Let $v$ and $B$ be a vertex and base of $\Omega$, so that $\Omega$ is the join of $v$ and $B$. If $B$ is an $(n-1)$-simplex then we are done. Otherwise by induction $B$ contains a minimal cone of dimension less than $n-1$. Let $C_e(I)$ be a minimal cone in $B$ of dimension less than $n-1$. Since $e$ is an extreme point of $B$ it is also an extreme point of $\Omega$. Let $I'$ denote the join of $v$ and $I$. Since $I$ is the relative interior of a face of $B$, $I'$ is the relative interior of a face of $\Omega$ with dimension less than $n-1$.

Because any line from $e$ to $I$ is in $B$ it follows that any line from $e$ to $I'$ is in $\Omega$. Hence $C_e(I')$ is a cone in $\Omega$ with dimension less than $n$. If $C_e(I')$ is minimal then we are done otherwise we keep decreasing the dimension until we get a minimal cone in $\Omega$ of dimension less than $n$.
\end{proof}

\section{More on Focusing Points}

We have already seen that an isometry with a focusing $n$-simplex is a projective transformation. In this section we will prove a theorem that generalizes this idea by showing that an isometry with a focusing point is projective. This is because every extreme point of $\Omega$ must then be a focusing point and therefore $\Omega$ contains a focusing $n$-simplex. Before we prove theorem 13.2 we need the following result about convex domains in dimension 3.

\begin{lem}
If $\Omega$ is a 3-dimensional convex domain which does not contain an extreme line then $\Omega$ is either the join of a line with a rectangle or $\Omega$ is a cone.
\end{lem}
\begin{proof}
Let $\Delta$ be an extreme 3-simplex in $\Omega$ with vertices $e_1,...,e_4$. Define $P_i$ to be the plane which contains all of the vertices of $\Delta$ except the $i$th vertex. We can assume that $\Omega$ is not a 3-simplex thus we can find an extreme point $e$ of $\overline{\Omega}$ which is not a vertex of $\Delta$.

Suppose $e$ does not lie in any of the planes $P_i$. Without loss of generality we can assume that $e$ lies in the open halfspace of $\mathbb{R}^3$ determined by $P_1$ which does not contain $\Delta$. If the triangle $(e_2,e_3,e_4)$ is contained in $\partial\Omega$ then there exists a facet $F$ of $\Omega$ containing it. By theorem 5.8 in ~\cite{Convex} $F$ is an exposed face of $\Omega$ and so there exists a proper supporting hyperplane $H$ such that $F=H\cap\overline{\Omega}$. From theorem 4.5 in ~\cite{Convex} we have a contradiction because $\overline{\Omega}$ is the intersection of its supporting halfspaces.

Therefore $(e_2,e_3,e_4)$ is contained in $\Omega$. The line $L\subset\partial\Omega$ through $e_1$ and $e$ intersects $P_1$ in a single point $p\in\partial\Omega$. Because $[e_2,e_3]$, $[e_2,e_4]$ and $[e_3,e_4]$ are in $\partial\Omega$ it follows that $P_1\cap\partial\Omega$ is the boundary of $(e_2,e_3,e_4)$. Without loss of generality we can assume that $p\in(e_2,e_3)$. This is a contradiction because this means that $e$ lies in $P_4$.

So far we have shown that any extreme point of $\Omega$ must lie in one of the planes determined by the faces of $\Delta$. If every extreme point besides $e_1,e_2,e_3,e_4$ of $\Omega$ lies in the same plane determined by the same face of $\Delta$ then $\Omega$ is a cone.

Otherwise $\Omega$ has a facet $Q$ with at least four extreme points $q_1,q_2,q_3,q_4$ which form a quadrilateral. Furthermore we can find two extreme points $f_1$ and $f_2$ with $[f_1,f_2]\subset\partial\Omega$ and $f_1$ and $f_2$ do not lie in the plane containing $Q$. If $\Omega$ has a fifth extreme point $q$ contained in $Q$ then $\Omega$ contains the interior of the convex hull of a pentagon with a line. This convex hull contains an extreme line. Thus $Q$ is the convex hull of $q_1,q_2,q_3,q_4$.

If there is a third extreme point $f_3$ of $\Omega$ which does not lie in the plane containing $Q$ then $\Omega$ contains the interior of the convex hull of the triangle $(f_1,f_2,f_3)$ and $Q$ and therefore contains an extreme line. Thus there are at most two extreme points of $\Omega$ which are not contained in $Q$ and the result follows.
\end{proof}

\begin{thm}
If $\tau:\Omega\rightarrow\Omega'$ is an isometry that has a focusing point, then $\tau$ is a projective transformation.
\end{thm}
\begin{proof}
If $\Omega$ and $\Omega'$ are 2-dimensional convex sets then we have already seen that this is true in section 11.

\textit{Case 1:} $dim(\Omega)=dim(\Omega')=3.$ We will first consider the case when $\Omega$ and $\Omega'$ are 3-dimensional convex domains. In this case we can assume that $\Omega$ and $\Omega'$ are cones or the join of a line and a quadrilateral by lemmas 10.8, 13.1 and theorem 10.9. But the join of a quadrilateral and a line is not isometric to a cone by applying theorem 8.11. Thus from theorem 5.1 we can assume that $\Omega$ and $\Omega'$ are cones. Since we can assume that $\Omega$ and $\Omega'$ are not 3-simplices they must have a unique vertex and base. Furthermore $\Omega$ and $\Omega'$ are not the join of two skew line segments. It follows from theorem 8.11 that $\tau$ sends a line from the vertex of $\Omega$ to the base of $\Omega$ to a line from the vertex of $\Omega'$ to the base of $\Omega'$.

\textit{Claim: The vertex of} $\Omega$ \textit{is a focusing point of} $\tau.$ Let $e$ be a focusing point of $\tau$ and let $f$ be the extreme point of $\Omega'$ which $e$ focuses to under $\tau$. If the vertex of $\Omega$ is not a focusing point of $\tau$ then $e$ must be contained in the relative boundary of the base of $\Omega$.  Let $\alpha$ be a point different from $e$ in the relative boundary of the base of $\Omega$ such that $(e,\alpha)$ is contained in the interior of the base of $\Omega$. If $v$ is the vertex of $\Omega$ then the triangle with vertices $e,v$ and $\alpha$ is a 2-dimensional cross section of $\Omega$ rigid under $\tau$ which contains a focusing point. It follows that $\tau$ restricted to this triangle is a projective transformation. Thus two lines in this triangle from $v$ to the base of $\Omega$ are mapped to two lines from the vertex of $\Omega'$ to the base of $\Omega'$ with the same orientation. This contradicts the assumption that $v$ is not a focusing point of $\tau$, see lemma 10.2.

\textit{Claim: Any extreme point of} $\Omega$ \textit{is a focusing point of} $\tau.$ So far we have shown that we can assume that $v$ is a focusing point of $\tau$ that focuses to the vertex $v'$ of $\Omega'$. Now we will use $e$ to denote any extreme point of $\Omega$ different from $v$ and will show that $e$ is a focusing point of $\tau$. We can assume that $\Omega$ and $\Omega'$ are not 3-simplices because we know this result is already true in that case. Because $\Omega$ is a 3-dimensional cone which is not a 3-simplex there exists a minimal cone $C_e(I)$ in $\Omega$ with dimension $1$. This cone is just a triangle and two of its vertices are $e$ and $v$. Because $v$ is a focusing point of $\tau$ the restriction of $\tau$ to $C_e(I)$ is a projective transformation and $\tau(C_e(I))$ is a triangle in $\Omega'$. Furthermore if $r$ is a rigid ray in $\Omega$ with $a(r)=e$ then $a(\tau(r))$ is an extreme point of $\tau(C_e(I))$. Since $C_e(I)$ is a minimal cone it follows from corollary 8.10 and lemma 8.3 that $a(\tau(r))$ is an extreme point of $\Omega'$ as well and thus $e$ is a focusing point of $\tau$ from lemma 10.2.

We have show that every extreme point of $\Omega$ is a focusing point of $\tau$ and therefore $\tau$ contains a focusing 3-simplex and is a projective transformation by lemma 10.7.

\textit{Case 2:} $dim(\Omega)=dim(\Omega')>3.$ We will use the 2 and 3 dimensional cases to finish the proof in higher dimensions. Let $\Omega$ and $\Omega'$ be general convex domains with dimension $n>3$.

\textit{Claim: Any extreme point of} $\Omega$ \textit{is a focusing point of} $\tau.$ Suppose $v$ is a focusing point of $\tau$ and $f$ is the point to which $v$ focuses to under $\tau$ and let $e$ be any other extreme point of $\Omega$. Choose some minimal cone $C_e(I)$ in $\Omega$. If $L$ is a line in $C_e(I)$ from $e$ to $I$ then the accumulation points of $\tau(L)$ in $\partial\Omega'$ lie in the relative interiors $I_1$ and $I_2$ of two faces of $\Omega'$ such that if $r$ is a ray in $\Omega$ with $a(r)=e$ then $a(\tau(r))$ is in $I_1$. Also $\tau(C_e(I))=J$ where $J$ is the join of $I_1$ and $I_2$, see theorem 8.11, and the relative boundary of $J$ is contained in $\partial\Omega'$ because $C_e(I)$ is minimal, see lemma 12.1.

If $I_1$ is an extreme point of $\Omega'$ then we are done by lemma 10.2 so suppose otherwise. Then there exists an extreme point $e'$ of $\Omega'$ in the relative boundary of $I_1$. Let $p$ be any point in $I_2$ and $\alpha$ be a point in the relative boundary of $I_1$ which lies in the relative interior of an opposite face in $\overline{I_1}$ of $e'$. By construction the triangle $T$ with vertices $\alpha,e'$ and $p$ is contained in $\Omega'$, see figure 12 below. Furthermore its relative boundary is contained in the boundary of $\Omega'$ because the relative boundary of $T$ is contained in the relative boundary of $J$.

The isometry $\tau^{-1}$ maps any two lines from $p$ to $(\alpha,e')$ to lines from $I$ to $e$ with the same orientation by lemma 8.3 and corollary 8.10. It follows that $\tau^{-1}_T$ is a non projective isometry of two dimensional triangles with their Hilbert metrics. Therefore $f$ does not lie in the plane containing $T$, because $f$ is focusing point of $\tau^{-1}$. Thus $(\alpha,e',p,f)$ is a 3-simplex in $\Omega'$. If $D'=P'\cap\Omega'$ is the 3-dimensional cross section of $\Omega'$ containing this 3-simplex then we will show that $\tau^{-1}(D')=P\cap\Omega$ for some 3-dimensional affine subspace $P$.

\textit{Main Claim:} $\tau^{-1}(D')=P\cap\Omega$ \textit{for some 3-dimensional affine subspace} $P$. In order to prove the main claim we need to understand which part of the boundary of the 3-simplex $[\alpha,e',p,f]$ lies in $\partial\Omega'$.

\textit{Claim:} $[\alpha,e',f]\subset\partial\Omega'$. If the triangle $(\alpha,e',f)$ is contained in $\Omega'$ then there is a straight line in $\Omega'$ from $f$ to $(\alpha,e')$. From corollary 8.10 the isometry $\tau^{-1}$ maps this line to an extreme line from $e$ to $v$. This contradicts the assumption that $e$ is not a focusing point of $\tau$ so the closed triangle $[\alpha,e',f]$ is contained in the boundary of $\Omega'$.

\textit{Claim:} $[e',p,f]\subset\partial\Omega'$. If the triangle $(e',p,f)$ is contained in $\Omega'$ then let $L_1$ and $L_2$ be two lines in $\Omega'$ from $f$ to $(e',p)$ that converge to two different points $\beta_1$ and $\beta_2$ in $(e',p)$, see figure 12. Furthermore let $M_1$ and $M_2$ be the two lines in $T$ from $\alpha$ to $\beta_1$ and $\beta_2$ respectively. Because $\tau^{-1}$ restricted to $T$ is a non projective isometry of triangles, if $r_1$ and $r_2$ are rays on $M_1$ and $M_2$ converging to $\beta_1$ and $\beta_2$ respectively then $a(\tau^{-1}(r_1))=a(\tau^{-1}(r_2))$.

Because $f$ is a focusing point of $\tau^{-1}$ the triangle $(\alpha,f,\beta_1)$ is mapped by $\tau^{-1}$ into a 2-dimensional plane $P_1$. Let $r_1'$ and $r_2'$ be rays on $L_1$ and $L_2$ converging to $\beta_1$ and $\beta_2$. By continuity $\tau^{-1}(r_1)$ and $\tau^{-1}(r_1')$ are contained in $P_1$ as well. It follows from corollary 8.10 that $a(\tau^{-1}(r_1))$ and $a(\tau^{-1}(r_1'))$ are contained in the relative interior of the same face of $P_1\cap\Omega$. If $\gamma$ is a path in $(\alpha,f,\beta_1)$ from $r_1$ to $r_1'$ then $\tau^{-1}(\gamma)$ is a path from $\tau^{-1}(r_1)$ to $\tau^{-1}(r_1')$ and every point on $\tau^{-1}(r_1)$ except the endpoints if contained in $\tau^{-1}(\alpha,f,\beta_1)$.

From continuity it follows that $a(\tau^{-1}(r_1'))=a(\tau^{-1}(r_1))$ and by similar reasoning $a(\tau^{-1}(r_2'))=a(\tau^{-1}(r_2))$. Thus $\tau^{-1}(L_1)=\tau^{-1}(L_2)$ because $f$ is a focusing point of $\tau^{-1}$ and we have already shown $a(\tau^{-1}(r_1))=a(\tau^{-1}(r_2))$. This is a contradiction because $L_1$ and $L_2$ were chosen to be distinct lines in $\Omega'$ so the triangle $[e',p,f]$ is contained in $\partial\Omega'$.

\textit{Claim:} $\tau^{-1}(\alpha,e',p,f)$ \textit{is an open 3-simplex}. Recall that $P'$ is the 3-dimensional affine subspace containing the interior of the 3-simplex $[\alpha,e',p,f]$. Because $\tau^{-1}$ of $(\alpha,e',p)$ is a triangle and $f$ is a focusing point of $\tau^{-1}$ we have that $\tau^{-1}(\alpha,e',p,f)$ is an open 3-simplex $\Delta$ that is contained in a 3-dimensional affine subspace $P$.

We will now complete the proof of the main claim. Let $H_f$ be the closed halfspace in $P'\cap\Omega'$ bounded by $[\alpha,e',p]$ which does not contain $[\alpha,e',p,f]$ and $H_{e'}$ be the closed halfspace in $P'\cap\Omega'$ bounded by $[\alpha,p,f]$ which does not contain $[\alpha,e',p,f]$. There are three possibilities for points $x$ in $D'=P'\cap\Omega'$ which are not in $(\alpha,e',p,f)$. We want to show that in all three cases $\tau^{-1}(x)\in P\cap\Omega$ to conclude that $\tau^{-1}(D')\subset P\cap\Omega$.

\textit{Case 1:} $x\in H_f\setminus H_{e'}$. If $x$ lies in $H_f$ but not $H_{e'}$ the $x$ lies on a rigid line in $\Omega'$ with $f$ as an accumulation point and this line passes through $(\alpha,e',p,f)$ and thus $\tau^{-1}(x)$ is contained in $P\cap\Omega$.

\textit{Case 2:} $x\in H_{e'}\setminus H_f$.
Similarly to case 1, the point $\tau^{-1}(x)$ is contained in $P\cap\Omega$ if $x$ lies in $H_{e'}$ but not $H_f$ because $x$ lies on a rigid line in $\Omega'$ that passes through $(\alpha,e',p,f)$  and has $e'$ as one of its accumulation points.

\textit{Case 3:} $x\in H_f\cap H_{e'}$.
The final case to consider if when $x$ lies in both $H_f$ and $H_{e'}$. There exists a rigid line in $\Omega'$ through $x$ with accumulation point $f$ that contains a point in $H_{e'}\cap H_f^c$ and therefore $\tau^{-1}(x)$ is contained in $P\cap\Omega$ because we have $\tau^{-1}(H_{e'}\cap H_f^c)\subset P\cap\Omega$ from the previous case.

We have shown that $\tau^{-1}(D')\subset P\cap\Omega$ and in a similar manner as in the proof of lemma 9.4 we get $\tau^{-1}(D')= P\cap\Omega$ which completes the proof of the main claim. Therefore $D'$ is a cross section of $\Omega'$ which is rigid under $\tau$ and so $\tau:P\cap\Omega\rightarrow D'$ is an isometry of 3-dimensional convex domains which has a focusing point $v$. It follows from the 3-dimensional case that $\tau|_{P\cap\Omega}$ is projective. This is a contradiction because $e$ is in the relative boundary of $P\cap\Omega$ and we assumed $e$ is not a focusing point of $\tau$ at the beginning of case 2 of the proof.

So we have shown that because $\tau$ has a focusing point every extreme point of $\Omega$ is a focusing point of $\tau$ and therefore $\tau$ is a projective transformation.

\end{proof}


\begin{figure}[h]
    \centering
    \def\svgwidth{\columnwidth}
    \begin{overpic}[]{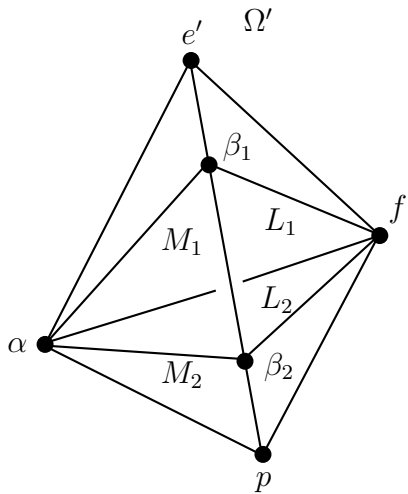}
        \put (50,105) {$\Omega'$}
        \put (-7,27.5) {$\alpha$}
        \put (35,102) {$e'$}
        \put (85,60) {$f$}
        \put (53,-5) {$p$}
        \put (45,75) {$\beta_1$}
        \put (55,22) {$\beta_2$}
        \put (55,57) {$L_1$}
        \put (54,38) {$L_2$}
        \put (30,52) {$M_1$}
        \put (30,20) {$M_2$}
    \end{overpic}
    \vspace{3mm}
    \caption{Theorem 13.2}
\end{figure}

\section{Isometries of Convex Cones}

 In the proof of theorem 13.2 we saw that cones play an important role in understanding isometries of the Hilbert metric. In this section we will study the Hilbert isometries of general convex cones and develop some tools that will allow us to understand the Hilbert isometries of any convex domain.

In this section the term cone will mean a convex domain $\Omega$ in $\mathbb{R}^n$ which consists of all straight lines from a $n-1$ dimensional convex domain $B$ in $\mathbb{R}^n$ to a point $v$ which does not lie in the hyperplane containing $B$. We will call $v$ a vertex of $\Omega$ and $B$ will be called a base of $\Omega$ corresponding to $v$. The main lemma shows that a cone with the Hilbert metric is not isometric to a convex domain with the Hilbert metric which is not a cone.

\begin{lem}
Suppose $\tau:\Omega\rightarrow\Omega'$ is an isometry between Hilbert domains and $\Omega$ is a cone, then $\Omega'$ is a cone. Furthermore if $L$ is a line in $\Omega$ from a vertex $v$ of $\Omega$ to a corresponding base $B$ of $\Omega$ then $\tau(L)$ is a line in $\Omega'$ from a vertex $v'$ of $\Omega'$ to a corresponding base $B'$ of $\Omega'$.
\end{lem}
\begin{proof}
Let $v$ and $B$ be a vertex and corresponding base of $\Omega$. If $v$ is a focusing point of $\tau$ then $\tau$ is a projective transformation from theorem 13.2 so the conclusion of the lemma follows immediately from lemma 10.2 and corollary 8.10.

Otherwise let $L$ be a rigid line in $\Omega$ from $v$ to $B$. If $r$ is a rigid ray on $L$ with $a(r)=v$ then $a(\tau(r))$ is contained in the relative interior $B'$ of a proper face of $\Omega'$ with $dim(B')\geq1$. Let $r_B$ be a ray on $L$ with $a(r_B)$ in $B$. The accumulation point $a(\tau(r_B))$ lies in the relative interior $v'$ of a proper face of $\Omega'$. It follows from corollary 8.10 that the image of any line from $v$ to $B$ in $\Omega$ under $\tau$ is a line from $B'$ to $v'$ in $\Omega'$ and $\Omega'$ is the join of $B'$ and $v'$ by theorem 8.11.

We will use strong induction to finish the proof of this lemma. The base case is when $dim(\Omega)=dim(\Omega')=2$. The only convex cone in dimension 2 is a triangle and we have already seen that a triangle is only isometric to another triangle so the lemma follows in this case from our previous results. Thus we can assume that $dim(\Omega)=dim(\Omega')=n>2$ and the statement of the lemma holds for any isometry of a convex cone of dimension less than $n$.\

If $v'$ is an extreme point of $\Omega'$ then $\Omega'$ is a cone with vertex $v'$ and corresponding base $B'$ so we are done. If not then $dim(v')\geq1$ and since we can assume $\Omega'$ is not an $n$-simplex it follows from lemma 12.3 that there exists a minimal cone $C_{e'}(I')$ in $\Omega'$ with $dim(C_{e'}(I'))=m<n$. Because $\Omega'$ is the join of $B'$ and $v'$ there are two cases to consider.

For the first case suppose $e'$ is contained in the relative boundary of $v'$. It follows from lemma 12.1 that $C_{e'}(I')$ is a rigid cross section of $\tau^{-1}$ so $\tau^{-1}:C_{e'}(I')\rightarrow\tau^{-1}(C_{e'}(I'))$ is a Hilbert isometry. It follows from the induction hypothesis that $\tau^{-1}(C_{e'}(I'))$ is an $m$-dimensional cone contained in $\Omega$ with vertex $e$ and corresponding base $I$ with the property that a line from $e'$ to $I'$ is mapped to a line from $e$ to $I$ but possibly with reverse orientation.
Lemma 12.1 implies that the relative boundary of $\tau^{-1}(C_{e'}(I'))$ is contained in $\partial\Omega$.

We will show that $\tau^{-1}(C_{e'}(I'))=C_e(I)$ is a minimal cone in $\Omega$. To do this we need to show that $I$ is the relative interior of a face of $\Omega$ and $e$ is an extreme point of $\Omega$. $I$ is contained in the relative interior $U$ of some face of $\Omega$ by corollary 8.10. If there exists a point $\alpha\in U\setminus I$ the line from $e$ to $\alpha$ is mapped to a line from $e'$ to $I'$ (possibly with reverse orientation) by corollary 8.10. This is impossible because a line that is not in $\tau^{-1}(C_{e'}(I'))$ is mapped inside of $C_{e'}(I')$. Thus $U=I$ and $I$ is the relative interior of a face of $\Omega$.

If $e$ is not an extreme point of $\Omega$ then there exists an open segment $S\subset\partial\Omega$ containing $e$. Let $L_1$ be a rigid line in $\Omega$ from a point in $S$ different from $e$ to a point in $I$. Then $\tau(L_1)$ is a rigid line in $C_{e'}(I')$ from $e'$ to $I'$ by corollary 8.10. This is a contradiction because $L_1$ is not contained in $\tau^{-1}(C_{e'}(I'))$. Hence $e$ is an extreme point of $\Omega$ and $C_{e}(I)$ is a minimal cone in $\Omega$.

We have already seen that a rigid line in $\Omega'$ from $e'$ to $I'$ is mapped by $\tau^{-1}$ to a rigid line in $\Omega$ from $e$ to $I$. The orientation of these lines must be reversed because we assumed that $v$ is not a focusing point of $\tau$ and therefore $\tau$ has no focusing point, see theorem 13.2.

We claim that because $e'$ is in the relative boundary of $v'$ and $C_{e'}(I')$ is a minimal cone in $\Omega'$ then $I'$ is the join of $B'$ with the relative interior of a proper face of $v'$. Clearly $I'$ is not equal to $B'$ or $v'$ because any line from $e'$ to either of these sets is in $\partial\Omega'$ whereas lines from $e'$ to $I'$ are in $\Omega'$. Furthermore $I'$ is not the join of $v'$ with the relative interior of a proper face of $B'$ because this join is contained in $\partial\Omega'$ and $e'$ is contained in the relative boundary of this set. Thus $I'$ must be the join of the relative interior of a proper face $I'_{v'}$ of $v'$ with the relative interior $I'_{B'}$ of a nonempty face of $B'$. We need to show that $I'_{B'}=B'$. If $\beta'$ is a point in $I'_{B'}$ and $\alpha'$ is a point in $I'_{v'}$ then $(\alpha',\beta')\subset I'$ and so the triangle $(e',\alpha',\beta')$ is contained in $\Omega'$. Hence any line from $\beta'$ to a point on $(e',\alpha')$ is contained in $\Omega'$. Any line from a point in the relative boundary of $B'$ to $\overline{v'}$ is contained in $\partial\Omega'$ so $\beta'\in B'$ and the claim follows. Furthermore the line $(e',\alpha')$ must be in $v'$.

The relative boundary of $(e',\alpha',\beta')$ is contained in $\partial\Omega'$ and $\tau^{-1}(e',\alpha',\beta')$ is contained in a 2 dimensional plane thus $\tau^{-1}(e',\alpha',\beta)$ is a triangle and $\tau^{-1}:(e',\alpha',\beta')\rightarrow\tau^{-1}(e',\alpha',\beta')$ is a Hilbert isometry of triangles. This is a contradiction because the family of lines through $\beta'$ in $(e',\alpha',\beta')$ is mapped by $\tau^{-1}$ to the family of lines through $v$ in $\tau^{-1}(e',\alpha',\beta')$ with the same orientation by corollary 8.10. But we have already seen that the lines from $e'$ to $(\alpha',\beta')\subset I'$ in $(e',\alpha',\beta)$ are mapped by $\tau^{-1}$ to lines from $e$ to $I$ in $\Omega$ with the orientation reversed. This is impossible because an isometry of triangles either preserves all orientations of lines through vertices are reverses all orientations of such lines. This concludes the first case.

For the second case assume that $e'$ is in the relative boundary of $B'$. By a similar argument as before we can show that $I'$ is the join of $v'$ with the relative interior $I'_{B'}$ of a proper face of $B'$. Furthermore $e'$ and $I'_{B'}$ are opposite in $B'$ because lines from $e'$ to $I'$ are in $\Omega$. As before $\tau^{-1}(C_{e'}(I'))=C_{e}(I)$ is a minimal cone and $\tau^{-1}$ maps lines from $e'$ to $I'$ to lines from $e$ to $I$ with the reverse orientation. Let $\beta'$ be a point in $I'_{B'}$ so we have $(e',\beta')\subset B'$. Furthermore let $f$ be an extreme point in the relative boundary of $v'$ and $\alpha'$ a point in the relative boundary of $v'$ so that $(f,\alpha')\subset v'$. Because lines from $B'$ to $v'$ are rigid proposition 3.3 implies that $(e',\beta',f,\alpha')$ is a 3-simplex in $\Omega'$ and one can check that the relative boundary of this 3-simplex is contained in $\partial\Omega'$, see figure 13. If $p\in(e',\beta')$ then $(p,f,\alpha')$ is a cross section of $\Omega'$ which is rigid under $\tau^{-1}$ and thus $\tau^{-1}|_{(p,f,\alpha')}$ is a projective transformation of triangles.

The triangle $\tau^{-1}(p,f,\alpha')$ and $e$ determine a 3-dimensional plane $P$. By construction any line $L'$ from $e'$ to $(\beta',f,\alpha')$ is in $C_{e'}(I')$ and $L'$ intersects $(p,f,\alpha')$ so $\tau^{-1}(L')\subset P\cap\Omega$. So we have $\tau^{-1}(e',\beta',f,\alpha')\subset P\cap\Omega$. As shown previously this means $\tau^{-1}(e',\beta',f,\alpha')= P\cap\Omega$ and so $\tau^{-1}|_{(e',\beta',f,\alpha')}$ is a Hilbert isometry of 3-simplices. This is impossible because a line from $p$ to $(f,\alpha')$ is mapped by $\tau^{-1}$ to a line from $v$ to $B$ with the same orientation by corollary 8.10 and an isometry of 3-simplices always maps a line through a vertex to a line through a vertex.

\end{proof}

\begin{figure}[h]
    \centering
    \def\svgwidth{\columnwidth}
    \begin{overpic}[]{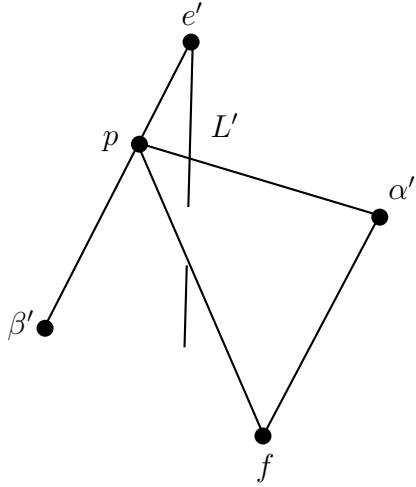}
        \put (-7,27.5) {$\beta'$}
        \put (35,102) {$e'$}
        \put (85,60) {$\alpha'$}
        \put (53,-7) {$f$}
        \put (42,75) {$L'$}
        \put (16,73) {$p$}
    \end{overpic}
    \vspace{3mm}
    \caption{Lemma 14.1}
\end{figure}

\begin{cor}
Suppose that $\tau:\Omega\rightarrow\Omega'$ is an isometry and $C_e(I)$ is a minimal cone in $\Omega$ then $\tau(C_e(I))$ is a minimal cone in $\Omega'$. Furthermore there exists an extreme point $e'$ of $\Omega'$ and the relative interior $I'$ of a proper face of $\Omega'$ such that $\tau(C_e(I))=C_{e'}(I')$ and $\tau$ maps a line from $e$ to $I$ to a line from $e'$ to $I'$.
\end{cor}
\begin{proof}
The proof is the same as that used in the proof of lemma 14.1 to show that $\tau^{-1}(C_{e'}(I'))=C_e(I)$ is a minimal cone in $\Omega$.
\end{proof}


\begin{lem}
Suppose $\tau:\Omega\rightarrow\Omega'$ is an isometry of convex domains and $\Omega$ is a homogeneous cone, then $\Omega'$ is a homogeneous cone.
\end{lem}
\begin{proof}
$\Omega'$ is a cone by lemma 14.1. If $x$ and $y$ are two points in $\Omega'$ then there exists a projective transformation $p\in PGL(\Omega)$ such that $p\tau^{-1}(x)=\tau^{-1}(y)$. Thus $\tau p\tau^{-1}(x)= y$. The result follows because $\tau p\tau^{-1}$ has a focusing point and is therefore in $PGL(\Omega')$.
\end{proof}

\section{Isometries of General Convex Domains}

In this section we will study the isometries of general convex domains. We will prove $PGL(\Omega)$ is at most an index 2 subgroup of $Isom(\Omega)$ for any convex domain. This extends the results of Lemmens and Walsh on polyhedral domains to arbitrary convex domains. This has also be shown to be true using different methods in ~\cite{Walsh2}.

\begin{lem}
If $\tau:\Omega\rightarrow\Omega$ is an isometry then $\tau^2$ is a projective transformation.
\end{lem}
\begin{proof}
Suppose $e$ is an extreme point of $\Omega$ and $C_e(I)$ is a minimal cone in $\Omega$ with vertex $e$. Using the setup in the statement of corollary 14.2 with $\tau^2$ replacing $\tau$ we have $\tau^2(C_e(I))=C_{e'}(I')$ and $\tau$ maps a line from $e$ to $I$ to a line from $e'$ to $I'$ with the same orientation. The orientation is the same because $\tau$ either changes or reverses orientation so applying $\tau$ twice gives the same orientation. Thus $e$ is a focusing point of $\tau^2$ by lemma 10.2 and it follows that $\tau^2$ is a projective transformation.
\end{proof}

We will now prove the main result of this thesis which shows that $PGL(\Omega)$ is at most an index 2 subgroup of $Isom(\Omega)$. This theorem was also established by Walsh in ~\cite{Walsh2}.

\begin{thm}
If $\Omega$ is an open bounded convex set in $\mathbb{R}^n$ then $PGL(\Omega)$ is a normal subgroup of $Isom(\Omega)$ and has index at most 2, that is

\noindent $PGL(\Omega)=Isom(\Omega)$ or
\begin{displaymath}Isom(\Omega)/PGL(\Omega)\cong\mathbb{Z}_2.\end{displaymath}
\end{thm}
\begin{proof}
Suppose $\tau_1$ and $\tau_2$ are isometries of $\Omega$ which are not projective transformations and let $C_e(I)$ be a minimal cone in $\Omega$. By corollary 14.2 we have $\tau_2(C_e(I))=C_{e_2}(I_2)$ is a minimal cone in $\Omega$ and $\tau_2$ maps lines from $e$ to $I$ in $C_e(I)$ to lines from $e_2$ to $I_2$ with the reverse orientation. Again by corollary 14.2 $\tau_1^{-1}(C_{e_2}(I_2))=C_{e_1}(I_1)$ is a minimal cone in $\Omega$ and $\tau_1^{-1}$ maps lines from $e_2$ to $I_2$ in $C_{e_2}(I_2)$ to lines from $e_1$ to $I_1$ with the reverse orientation. Thus $\tau_1^{-1}\tau_2$ maps lines from $e$ to $I$ to lines from $e_1$ to $I_1$ with the same orientation and therefore $e$ is a focusing point of this isometry, see lemma 10.2. It follows that $\tau_2=\tau_1 p$ for some projective transformation $p$ of $\Omega$.

We have shown that $\tau_1 PGL(\Omega)=\tau_2 PGL(\Omega)$ and therefore there are only 2 cosets of $PGL(\Omega)$ in $Isom(\Omega)$.
\end{proof}

\section{Isometries of the Hilbert Metric in $\mathbb{R}^3$}

In this section we will study the isometries of 3-dimensional convex domains. We have seen the 3-simplex and the cone on a disk have isometries which are not projective transformations. We will give the proof of theorem 1.5 which tells us that an isometry between 3-dimensional convex domains which are not cones is a projective transformation.

\begin{thm}
Suppose $\Omega$ and $\Omega'$ are the interiors of compact convex sets in $\mathbb{R}^3\subset\mathbb{R}P^3$ which are isometric. If $\Omega$ is not a cone then $\Omega'$ is not a cone and $Isom(\Omega, \Omega')=PGL(\Omega,\Omega')$.
\end{thm}
\begin{proof}
Since $\Omega$ is not a cone lemma 14.1 implies that $\Omega'$ is not a cone. If $\Omega$ contains an extreme line the $\tau$ is projective. If not then lemma 13.1 implies $\Omega=\Omega'$ is the join of a line and a rectangle. It follows from theorem 5.1 that $\tau$ is a projective transformation.
\end{proof}

The author believes lemma 16.1 holds in all dimensions. More work is needed to establish this lemma in higher dimensions because there are more possibilities for convex domains which do not contain extreme lines in dimensions greater than 3.

\section{Concluding Remarks}

In this section we will give an overview of what is known about the isometries of the Hilbert metric and some questions that are still unanswered. The main result in this thesis was that for a general convex domain $PGL(\Omega)$ is at most an index 2 subgroup of $Isom(\Omega)$. We also gave a large class of domains for which these two groups are equal. This result was also proven independently by Walsh in ~\cite{Walsh2}. We state his result below.

\begin{thm}[Walsh]
Suppose $\Omega$ is a convex domain and $C_\Omega$ is a cone over $\Omega$. If $C_\Omega$ is a non Lorentzian symmetric cone then $PGL(\Omega)$ is an index 2 subgroup of $Isom(\Omega)$. Otherwise $PGL(\Omega)=Isom(\Omega)$.
\end{thm}

To prove this theorem Walsh used the detour metric and generalized other ideas in ~\cite{Polyhedral} to general convex domains. This theorem answers several questions posed by de la Harpe in ~\cite{Harpe} and shows that $Isom(\Omega)$ is always a Lie group.

In this thesis we have shown that in dimension 2 that two convex domains are isometric if and only if they are projectively equivalent. Furthermore in any dimension a convex domain with an extreme line can only be isometric to a convex domain containing an extreme line and strictly convex domains are only isometric to strictly convex domains. This leads us to make the following conjecture.

\begin{conj}
If $(\Omega,d_\Omega)$ and $(\Omega',d_\Omega')$ are Hilbert geometries then they are isometric if and only if they are projectively equivalent.
\end{conj}

\bibliographystyle{plain}
\bibliography{ThesisArxiv}


\end{document}